\documentclass[11pt]{amsart}

\usepackage{amsmath,amssymb,amsthm,amsfonts}

\theoremstyle{plain}
\newtheorem{theorem}{Theorem}[section]
\newtheorem{corollary}[theorem]{Corollary}
\newtheorem{proposition}[theorem]{Proposition}
\newtheorem{lemma}[theorem]{Lemma}

\theoremstyle{definition}
\newtheorem{remark}[theorem]{Remark}

\newtheorem{example}[theorem]{Example}

\newtheorem{definition}[theorem]{Definition}

\newcommand{\eps}{\varepsilon}
\newcommand\trnorm[1]{\left\|\kern-1.2pt\left|#1\right\|\kern-1.2pt\right|}

 \DeclareMathOperator{\sign}{sign}

 \DeclareMathOperator{\supp}{supp}

\begin{document}
\title[New formulas for decreasing rearrangements]{New formulas for decreasing rearrangements and a class of Orlicz-Lorentz spaces}

\author[Kami\'nska]{Anna Kami\'nska}
\address[Kami\'nska]{Department of Mathematics\\
University of Memphis\\ Memphis, USA}
\email{\texttt{kaminska@memphis.edu}}

\author[Raynaud]{Yves Raynaud}
\address[Raynaud]{Institut de Math\'ematiques de Jussieu, Universit\'e Paris 06-UPMC and CNRS, 4 place Jussieu, F-75252 Paris cedex 05, France}
\email{\texttt{yves.raynaud@upmc.fr}}

\keywords{decreasing rearrangement, Hardy's formulas, dual spaces, Orlicz-Lorentz spaces}

\thanks {2010\emph{ Mathematics Subject Classification.}\ {26D07, 39B62, 42B25, 46B10, 46E30}}

\date{\today}

\begin{abstract}   Using a nonlinear version of the well known Hardy-Littlewood inequalities, we derive new formulas for decreasing rearrangements of functions and sequences in the context of convex functions. We use these formulas for deducing several properties of the modular functionals defining the  function and sequence spaces $M_{\varphi,w}$ and $m_{\varphi,w}$ respectively, introduced earlier in \cite{HKM} for describing the K\"othe dual of ordinary  Orlicz-Lorentz  spaces  in a large variety of cases ($\varphi$ is an Orlicz function and $w$ a {\it decreasing} weight). We study these $M_{\varphi,w}$ classes in the most general setting, where they may even not be linear, and identify their K\"othe duals with  ordinary (Banach) Orlicz-Lorentz  spaces. We introduce a new class of rearrangement invariant Banach spaces $\mathcal{M}_{\varphi,w}$ which proves to be the K\"othe biduals of the $M_{\varphi,w}$ classes. In the case when the class $M_{\varphi,w}$ is a separable quasi-Banach space, $\mathcal{M}_{\varphi,w}$ is its Banach envelope.

\end{abstract}
\maketitle

\section{Introduction and Preliminaries}\label{prelim}

The theories of Lorentz, Orlicz-Lorentz or Marcinkiewicz spaces have been developed on the ground of the concept of decreasing rearrangement, submajorization and maximal functions. A basic tool in this domain are
the Hardy-Littlewood inequalities and equations which are thoroughly discussed in several papers and monographs \cite{BS, KPS}.  In this paper we study special Orlicz-Lorentz classes which were introduced for the purpose of expliciting the structure of the K\"othe duals of a large variety of "classical" Orlicz-Lorentz spaces \cite{HKM}. It turns out that in this context a certain kind of ``generalized Hardy-Littlewood inequalities'', in the spirit of those introduced in the fifties by G. G. Lorentz \cite{Lo}, are particularly useful.

Let's first introduce basic notions, definitions, symbols and facts needed later.
As usual by $\mathbb{R}$, $\mathbb{R}_+$ and $\mathbb{N}$ we denote the set of all real, non-negative real and natural numbers, respectively.
Let $I=(0,a), 0<a\le \infty$.    By $L^0$ denote the set of all Lebesgue measurable functions $f: I\rightarrow \mathbb{R}$.  Given $f\in L^0$ define its {\it decreasing rearrangement} as
\[
f^*(t)  = \inf |\{s\in I: d_f(s) \le t\}|, \ \ \ t\in I,
\]
where $d_f(s) = |\{t\in I: |f(t)| > s\}|$, $s\ge 0$.  Analogously, if
$x = \{x(n)\}$ is a sequence of real numbers then $x^* = \{x^*(n)\}$, where
\[
x^*(n)= \inf \{s>0: d_x(s) \le n-1\}, \ \ \  n\in \mathbb{N},
\]
 and $d_x(s) = |\{k\in \mathbb{N}: |x(k)| > s\}|$, $s \ge 0$.  Given two functions $f, g\in L^0$, or respectively two sequences $x, y$, we write $f\sim g$, respectively
$x\sim y$, whenever $f^*=g^*$, respectively $x^*=y^*$.

Given $f,g\in L^0$, we say that $g$ is {\it submajorized} by $f$ and write $g\prec f$, whenever
\[
\int_0^t g^* \le \int_0^t f^* \ \ \  \text{for all}\ \  t\in I.
\]
Similarly for sequences $x=\{x(n)\}$, $y=\{y(n)\}$ we write $y\prec x$, if
\[
\sum_{n=1}^m y^*(n) \le \sum_{n=1}^m x^*(n) \ \ \ \text{for all} \ \ \ m\in\mathbb{N}.
\]
By $|A|$ we shall denote the cardinality of $A\subset \mathbb{N}$
or the Lebesgue measure of $A$ for $A\subset I$. Recall that a function $\tau: I\rightarrow I$ is a {\it measure
preserving transformation} \cite{BS} if for every measurable set $A\subset
I$ the set $\tau^{-1}(A) = \{t\in (0,a): \tau(t) \in A\}$ is
measurable and $|\tau^{-1}(A)| = |A|$. It is well known that for any two measurable sets $A, B \subset I$ with $|A| = |B|$ there is a measure preserving transformation $\tau : A \to B$ which is measurable, one-to-one and onto function \cite{Roy}.
Any one-to-one and onto function of $\mathbb{N}$ will be called {\it automorphism} of $\mathbb{N}$, and obviously  it preserves the measure of $A\subset \mathbb{N}$.

The space $(E, \|\cdot\|_E)$ is called a (quasi) Banach function space if $E\subset L^0$, $\|\cdot\|_E$ is a (quasi) norm and whenever $f\in E$, $g\in L^0$ and
$|g| \le |f|$ then $g\in E$ and $\|g\|_E\le \|f\|_E$. We say that a Banach function space $E$ is a rearrangement invariant (r.i.) function space whenever  $f\in E$ yields that $f^*\in E$ and $\|f\|_E = \|f^*\|_E$.
We say that $(E, \|\cdot\|_E)$ satisfies the {\it Fatou property} if for any $f_n \in E$ if $f_n\uparrow f$ a.e. and $\sup_n \|f_n\|_E < \infty$ then $f\in E$ and $\|f_n\|_E \uparrow \|f\|_E$.  The space $E$ is said to be {\it order continuous} if $\|f_n\|_E \downarrow 0$ for any $f_n\in E$ such that $f_n\downarrow 0$ a.e. Analogously $(E, \|\cdot\|_E)$ is called a (quasi) Banach sequence space or rearrangement invariant sequence space if $E$ is a subspace of all real sequences and has the similar properties as analogous function spaces. For information on Banach function or sequence spaces we refer to \cite{BS, KA, KPS, LT2, Z} and for quasi-Banach spaces to \cite{KM, KPR, KR}.

 The terms decreasing or increasing will stand for non-increasing or non-decreasing, respectively.  A function $w\in L^0$ is called a {\it weight function} whenever it is non-negative and  decreasing.  We set $W(t)= \int_0^t w $ for all $t\in I$. The function $W$ is either everywhere infinite (except at 0) or everywhere finite on $I$.
  Similarly $w= \{w(n)\}$ is a {\it weight sequence} if it is non-negative  and
decreasing. Let also $W(n) = \sum_{i=1}^n w(i)$, $n\in \mathbb{N}$.  Notice that the function $W(t)/t$ on $I$, and the sequence  $\{W(n)/n\}$ are decreasing. It is said that the weight function or the weight sequence $w$ is {\it regular} if there exists $C>0$ such that $W(u) \le Cuw(u)$, for all $u\in I$ or $u\in \mathbb{N}$, respectively.

That weight functions and sequences are assumed to be decreasing will be recalled only in the statements of important results.

The symbol $\varphi$ throughout the paper stands for an {\it Orlicz function} that is  $\varphi:\Bbb R_+ \rightarrow \Bbb R_+$  is convex, strictly increasing and $\varphi(0)=0$.  It follows that $\varphi(t)/t$ is increasing and $\varphi(c/t)t$ is decreasing with respect to $t>0$, for any $c>0$. It is said that $\varphi$ satisfies {\it condition $\Delta_2$}
whenever there exists $K>0$
such that $\varphi(2u) \le K \varphi(u)$ for all $u\ge 0$.
Let $\varphi_*(t) = \sup_{s>0} \{st - \varphi(s)\}$, $t\ge 0$, be the {\it complementary} function of $\varphi$. An Orlicz function $\varphi$ is called {\it $N$-function} whenever $\lim_{t\to\infty} \varphi(t)/t = \infty$ and $\lim_{t\to 0} \varphi(t)/t = 0$. It is well known that the conjugate  $\varphi_*$ of an $N$-function $\varphi$ is also an $N$-function \cite{KrRut}. We say that two Orlicz functions $\varphi_1$ and $\varphi_2$ are equivalent if for some $C>0$, $\varphi_1(C^{-1}t) \le \varphi_2(t) \le \varphi_1(Ct)$ for all $t\ge 0$.

We also say that two expressions $A$ and $B$ (not both Orlicz functions) are equivalent, whenever there exists $C>0$ such that
$C^{-1}A \le B \le C A$.

In the context of the work presented here let's agree on the following convention. Given $f, v \ge 0$ on $I$, if $v(t) = 0$ then define

 \[
  \varphi\left(\frac{|f(t)|}{v(t)}\right) v(t)=  \begin{cases}
     &  0\ \  \text{if}\ \ f(t)=0; \\
    &  \infty\ \  \text{if}\ \ f(t)\neq 0.\\
   \end{cases}
  \]
Similarly as above, for sequences $v=\{v(n)\}\ge 0$ and $x=\{x(n)\}$,  if $v(n) =0$ define
 \[
  \varphi\left(\frac{|x(n)|}{v(n)}\right) v(n)=  \begin{cases}
     &  0\ \  \text{if}\ \ x(n)=0; \\
    &  \infty\ \  \text{if}\ \ x(n)\neq 0.\\
   \end{cases}
  \]
For any function $f\in L^0$ and sequence $\{x(n)\}$  let
\[
M(f) = \int_I \varphi\left(\frac{f^*}{w}\right) w \ \  \text{and}\ \ \
m(x) = \sum_{n=1}^\infty \varphi\left(\frac{x^*(n)}{w(n)}\right) w(n).
\]

Let's now discuss the results of this paper.
In section \ref{sec:New formulas} we discuss new rearrangement invariant formulas, expressing $M(f)$ or $m(x)$ in an equivalent way, in the spirit of Hardy-Littlewood formulas \cite[pp. 63-75]{KPS}, \cite[Chapter 2, sec. 2-3]{BS}.
Recall that the basic Hardy-Littlewood formulas state that for any $f,g\in L^0$, $\int_I f g \le \int_I f^*g^*$, and in fact we have that $\int_I f^* g^* = \sup_{g\sim v} \int_I |f| |v|$. There exists also a parallel formula if one of the function is increasing and the other is decreasing. So if $h\ge 0$ is increasing then $\int_I f^* h  = \inf_{g\sim h} \int_I |f| |g|$.  On the other hand G. G. Lorentz \cite{Lo} extended the Hardy-Littlewood inequalities as $\int_I \Phi(f ,g) \le \int_I \Phi(f^*,g^*)$, where the interval $I$ is finite and $\Phi: \mathbb R_+^2\to\mathbb R$ belongs to a certain class of continuous functions which are called nowadays ``supermodular''. This means that, for the natural lattice structure of  $\mathbb R^2$:
\[\forall x,y\in  \mathbb R_+^2,  \Phi(x\vee y)+\Phi(x\wedge y)\ge \Phi(x)+\Phi(y)\]
These ``generalized Hardy-Littlewood inequalities'' were extended in the last decade to $n$-dimensional context and to other kinds of rearrangement procedures, with applications to the theory of optimization and partial differential systems.

We show first that given an Orlicz function $\varphi$ and a weight function $w$,  $M(f) \le \int\varphi(|f|/v)v$ for any $v\sim w$, $v\ge 0$, and then that $M(f)= \inf_{v\sim w} \int\varphi(|f|/v)v $.  The work is conducted initially for finite sequences, then extended to infinite sequences and eventually to function case.  The link with the above cited result by G.G. Lorentz is the fact that the function $(s,t)\mapsto -\varphi(\frac st)t$ is supermodular on the interior of  $\mathbb R_+^2$. However it takes value $\infty$ on the semi-axis $\{t=0\}$. For this reason and for the commodity of the reader, our proof is self-contained and does not refer to \cite{Lo}. Next, under additional assumption of $w$ being "controlled by $\varphi$", in particular under regularity of $w$ and $\Delta_2$ condition of $\varphi$,  we refine the approximation of $M(f)$ by similar infimum  expression limited to positive $v$ and such that $v\sim w$.   We close the section showing another formula for $M(f)$ that states $M(f)= \inf_{v^*\le w} \int\varphi(|f|/v)v $. Analogous formulas are also proved for $m(x)$.

By $M_{\varphi,w}$  and $m_{\varphi,w}$ denote the class of all functions  $f\in L^0$ and
sequences $x=\{x(n)\}$, such that for some $\lambda>0$, $M(\lambda f)<\infty$ and
$ m(\lambda x)<\infty$, respectively. Given $f\in M_{\varphi,w}$ and $x\in m_{\varphi, w}$, define
\[
\|f\|_M = \inf\left\{\epsilon > 0: M\left(f/\epsilon\right) \le 1\right\}
 \ \ \ \text{and}\ \ \
 \|x\|_m = \inf\left\{\epsilon > 0: m\left(x/\epsilon\right) \le 1\right\}.
\]
 Note that these classes are never trivial.  They contain indicator functions of integrable sets, and more generally bounded measurable functions with supports of finite measure. Indeed the function $\varphi(1/w)w$ is increasing on $I$, and thus $\int_0^t\varphi(1/w)w<\infty$ for every $t\in I$. The notations $\|\cdot \|_M$ and $\|\cdot\|_m$ should not be misleading since these functionals are in general neither norms nor equivalent to norms. They are not even necessarily equivalent to quasi-norms, and the classes $M_{\varphi,w}$  and $m_{\varphi,w}$ may  not be even linear spaces. Nevertheless the set of positive decreasing elements in these classes is a convex cone, on which the functionals $\|\cdot \|_M$ and $\|\cdot \|_m$ are convex. Section \ref{sec:Classes M} is devoted to investigate these classes.

The classes $M_{\varphi,w}$ and $m_{\varphi,w}$ appeared for the first time in \cite{HKM} as duals of Orlicz-Lorentz spaces. It was proved that under assumptions of regularity of $w$, if $\varphi$ is an $N$-function then the K\"othe dual $(\Lambda_{\varphi,w})'$ of Orlicz-Lorentz space $\Lambda_{\varphi,w}$ coincides with $M_{\varphi_*,w}$ as sets and the homogeneous functional $\|\cdot\|_{M_{\varphi_*,w}}$ is equivalent to the dual norm in $(\Lambda_{\varphi,w})'$. In this section we study the more general case where $w$ is no more assumed to be regular, nor integrable on finite intervals. In this wider context it may happen that the  functionals $\|\cdot \|_M$ and $\|\cdot \|_m$ are not quasinorms.  We show that in this case the classes $M_{\varphi,w}$, $m_{\varphi,w}$ are not even closed under addition. This phenomenon was first described for the case of generalized Lorentz spaces in the paper \cite{CKMP}. A sufficient condition for these classes to be quasi-normed Banach spaces is that $1/w$ satisfies condition $\Delta_2$. This condition is also necessary when $\varphi$ has lower index  $\alpha_\varphi>1$. But even when this is not the case, it happens that these classes have nontrivial K\"othe duals provided $w$ is integrable on finite intervals, and these duals are precisely the ordinary Orlicz-Lorentz spaces $\Lambda_{\varphi_*,w}$ corresponding to the conjugate Orlicz function $\varphi_*$. This may be considered as a kind of generalization of the main result of \cite{HKM}.

In section \ref{sec:Spaces P}, inspired by the formulas proved in section \ref{sec:New formulas},  for any Orlicz function $\varphi$ and a weight $w$  we introduce a new class of function spaces $\mathcal{M}_{\varphi,w}$ and their sequence version the spaces $\mathfrak{m}_{\varphi,w}$. We first show that they are  rearrangement invariant Banach spaces with the Fatou property. It holds in general that $M_{\varphi,w} \subset \mathcal{M}_{\varphi,w}$ and $\mathfrak{m}_{\varphi,w} \subset \mathcal{M}_{\varphi,w}$. We prove that they are equal with the quasinorm $\|\cdot\|_M$ (resp., $\|\cdot\|_m$) and the norm in $\mathcal{M}_{\varphi,w}$ (resp., in $\mathfrak{m}_{\varphi,w}$) equivalent whenever $w$ is regular. The converse of this statement holds true under additional assumption that the lower index $\alpha_\varphi > 1$. The latter fact was obtained by  calculating the formula of the fundamental function of the spaces $\mathcal{M}_{\varphi,w}$ and $\mathfrak{m}_{\varphi,w}$. We finish this section by showing that the spaces $\mathcal{M}_{\varphi, w}$ are the K\"othe  duals of the ordinary Orlicz-Lorentz spaces $\Lambda_{\varphi_*,w}$, and thus the K\"othe biduals of the spaces $M_{\varphi,w}$  where this last identification is isometric.  It follows at once that if $M_{\varphi,w}$ is normable then $M_{\varphi,w}=\mathcal{M}_{\varphi,w}$ with equivalent norms.

The  identification of $\mathcal{M}_{\varphi,w}$ as K\"othe dual of  $\Lambda_{\varphi_*,w}$ was first given by K. Le\'snik, after attending a talk about   a first version of the present paper. Le\'snik's elegant proof is based on Calder\'on-Lozanovskii method and quite different from that given here. It will be presented in the paper \cite{KLR} where the theory of these new classes $\mathcal{M}_{\varphi,w}$ and $\mathfrak{m}_{\varphi,w}$ is developed further.

\section{Formulas for Decreasing Rearrangements}\label{sec:New formulas}

We define $[[1,n]] = \{1,\dots,n\}$ for any $n\in\mathbb{N}$.

\begin{proposition}\label{prop:Finite} Let $w=\{w(i)\}_{i=1}^n$ be a finite
decreasing, positive weight sequence. Then for every finite
sequence $x=\{x(i)\}_{i=1}^n$,
\begin{align*}
\sum_{i=1}^n \varphi\left(\frac{x^*(i)}{w(i)}\right) w(i)& = \inf
\left\{\sum_{i=1}^n \varphi\left(\frac{x^*(i)}{w(\sigma(i))}\right)
w(\sigma(i)): \ \ \sigma\ \text{is a permutation of}\
[[1,n]]\right\}
\end{align*}
\end{proposition}

\begin{proof}

Clearly, the left side of the required inequality is bigger than the right side.
It remains to show the opposite inequality. Let's start with two
dimensional spaces, and take $x^*=(s_1, s_2)$ and $w=(t_1, t_2)$,
where $s_1 > s_2 > 0$ and $ t_1 > t_2 >0$. Then we have $s_1/t_1 >
s_2/t_1,\, s_1/t_2 > s_2/t_2$ and $s_1/t_2 > s_1/t_1, \, s_2/t_2 >
s_2/t_1$. Consequently by convexity of $\varphi$ we get
\[
\frac{\varphi(s_1/t_1) - \varphi(s_2/t_1)}{s_1/t_1 - s_2/t_1}
\le \frac{\varphi(s_1/t_2) - \varphi(s_2/t_2)}{s_1/t_2 - s_2/t_2},
\]
which equivalently means that
\begin{equation}\label{eq:2m}
\sum_{i=1}^2 \varphi\left(\frac{x^*(i)}{w(i)}\right) w(i)
\le \sum_{i=1}^2 \varphi\left(\frac{x^*(i)}{w(\sigma(i))}\right)w(\sigma(i)),
\end{equation}
where $\sigma(1) = 2$ and $\sigma(2) = 1$. Now we reason by
induction on $n$.  We assume for $n>2$ that
\[
\sum_{i=1}^{n-1} \varphi\left(\frac{x^*(i)}{w(i)}\right) w(i) \le
\sum_{i=1}^{n-1} \varphi\left(\frac{x^*(i)}{w(\alpha(i))}\right) w(\alpha(i))
\]
for any permutation $\alpha$ of $[[1,n-1]]$.
If $\sigma(n)=n$, then $\sigma$ induces a
permutation of $[[1,n-1]]$ and we apply the induction hypothesis
to $\{x(i)\}_{i=1}^{n-1}$ and $\{w(i)\}_{i=1}^{n-1}$. Adding to both
sides the term $\varphi\left(\frac{x^*(n)}{w(n)}\right) w(n)$ we get the required inequality.
If $\sigma(n)<n$, then also $\sigma^{-1}(n)<n$. By (\ref{eq:2m})
setting $s_1 = x^*(\sigma^{-1} (n))$, $s_2 = x^*(n)$, $t_1=
w(\sigma(n))$, $t_2 = w(n)$ we have
\[
\varphi\left(\frac{x^*(\sigma^{-1}(n))}{w(n)}\right)
w(n)+\varphi\left(\frac{x^*(n)}{w(\sigma(n))}\right) w(\sigma(n))\ge
\varphi\left(\frac{x^*(\sigma^{-1}(n))}{w(\sigma(n))}\right)
w(\sigma(n))+\varphi\left(\frac{x^*(n)}{w(n)}\right) w(n).
\]
Hence if $\tau$ is the transposition which exchanges the indices
$n$ and $\sigma(n)$, that is if $\tau(i) = i$ for $i\ne n,
\sigma(n)$, and $\tau(n) = \sigma(n)$, $\tau(\sigma(n)) = n$,
then we get
\[ \sum_{i=1}^n
\varphi\left(\frac{x^*(i)}{w(\sigma(i))}\right) w(\sigma(i)) \ge
\sum_{i=1}^n \varphi\left(\frac{x^*(i)}{w(\tau\circ\sigma(i))}\right)
w(\tau\circ\sigma(i)).
\]
Finally since the permutation $\tau\circ\sigma$ fixes the index $n$, that
is $\tau\circ\sigma(n) = n$, we can apply the previous case and
the induction is finished.
\end{proof}

\begin{corollary}\label{cor:finite}
Let $w=\{w(i)\}_{i=1}^n$ be a finite, decreasing and positive sequence, and $(a_{ij})$ be a
$n\times n$ matrix of non-negative numbers satisfying the
condition
\[
 \sum_{j=1}^n a_{ij}=\sum_{j=1}^n a_{ji}  \  \ \ \  i= 1,\dots, n,
 \]
that is the $i$-th row and the $i$-th column have the same
sum,  then for every sequence  of real numbers $\{x(i)\}_{i=1}^n$ we
have
\begin{equation*}
\sum_{i,j=1}^n \varphi\left(\frac{x^*(i)}{w(i)}\right) w(i)a_{ij} \le
\sum_{i,j=1}^n \varphi\left(\frac{x^*(i)}{w(j)}\right) w(j)a_{ij}.
\end{equation*}

\end{corollary}

\begin{proof}

We first reduce the proof
to the case where $a_{ij}$ are positive rational numbers.
In fact we approximate $a_{ij}$, $1\le i\le n, 1\le j< n$, by rational
numbers $q_{ij}$ and then use the relations $\sum_j
q_{ij}=\sum_jq_{ji}$ for defining the last row of the
approximating matrix. Then by homogeneity we can further reduce
it to the case where $a_{ij}$ are natural numbers. We consider
then a partition of the interval $J=[[1,\sum_{ij}a_{ij}]]$ into
disjoint  intervals $I_{ij}\subset\Bbb N$, with the
respective lengths $|I_{ij}|=a_{ij}$. We put them in lexicographic
order
\[
I_{11}<I_{12}<...<I_{1n}<I_{21}<...<I_{2n}<...<I_{n1}<...<I_{nn},
\]
and define $A_i=\bigcup_j I_{ij}$ and $B_j=\bigcup_i I_{ij}$. Set
$\hat x(k)=x^*(i)$ and  $\hat w(k)=w(i)$ for $k\in A_i$. Note that
the sequences $\hat x$ and $\hat w$ are decreasing and
$|A_i|=|B_i|$.  So we can find a permutation $\sigma$ of $J$
mapping $A_i$ onto $B_i$ for each  $i=1,\dots,n$. Then the two
sides of the desired inequality are respectively equal to
$\sum_k\varphi\left(\frac{\hat x^*(k)}{\hat w(k)}\right) \hat w(k)$
and $\sum_k\varphi\left(\frac{\hat x^*(k)}{\hat w(\sigma(k))}\right)
\hat w(\sigma(k))$. Consequently,  by Proposition \ref{prop:Finite} the proof is finished.

\end{proof}

\begin{lemma}\label{l:Infinite:ineq} Let $w=\{w(n)\}$ be a decreasing weight sequence. Then for every sequence $x$,
\begin{align*}\label{eq:11}
\sum_{n=1}^\infty \varphi\left(\frac{x^*(n)}{w(n)}\right) w(n)& \le
\inf\left\{\sum_{n=1}^\infty \varphi\left(\frac{|x(n)|}{v(n)}\right)
v(n): v\sim w, v\ge 0\right\}\\
&\le \inf\left\{\sum_{n=1}^\infty \varphi\left(\frac{|x(n)|}{v(n)}\right)
v(n): v\sim w, v> 0\right\}\\
&\le \inf\left\{\sum_{n=1}^\infty \varphi\left(\frac{|x(n)|}{w\circ
\sigma(n)}\right) w\circ \sigma(n): \sigma
\ \text{automorphism of} \ \ \mathbb{N}\right\}.
\end{align*}

\end{lemma}

\begin{proof}
\par   Observe that the second inequality is obvious and the third one is also clear since for every
automorphism $\sigma$ of $\mathbb{N}$ we have $(w\circ \sigma)^* =
w$.

In order to prove the first one, let $v\ge 0$ be any sequence such that $v\sim w$.
We observe first that if $\supp w \neq \mathbb{N}$, then both sides of the first inequality are equal. In fact if $|\supp x| \le |\supp w|< \infty$, then the situation is reduced to the finite case as in Proposition \ref{prop:Finite}. If $|\supp x| > |\supp w|$ then $m(x) = \infty$, and for any $v\sim w$ there is at least one $i$ such that $v(i)=0$ and $x^*(i) > 0$. So the right side is also infinity. Thus we assume that $w> 0$.

We claim first that if $x(n)\nrightarrow 0$ as $n\rightarrow\infty$, then each expression
above is equal to infinity and the equalities hold.
Indeed if  $x(n)\nrightarrow 0$ then $\inf_n x^*(n) = K > 0$. Since $w$ is decreasing, there exists $b>0$ such that the set
$S= \{n: w(n) \le b\}$ is infinite. Hence using the fact that the function $t\mapsto\varphi(a/t)t$ is decreasing for a fixed $a$,
\[
\sum_{n=1}^\infty \varphi\left(\frac{x^*(n)}{w(n)}\right) w(n) \ge \sum_{n\in S} \varphi\left(\frac{x^*(n)}{w(n)}\right) w(n)
\ge \sum_{n\in S}  \varphi\left(\frac{K}{b}\right) b = \infty.
\]
Similarly, for any $v\sim w$, $v\ge 0$, the set $\{n: v(n) \le b\}$ is also infinite, and so
 $\sum_{n=1}^\infty \varphi\left(\frac{|x(n)|}{v(n)}\right)v(n) = \infty$. Therefore all expressions in the above inequalities are equal to infinity.

Suppose now that $x(n)\to 0$ if $n\to\infty$.
Letting $n\in\mathbb{N}$ be arbitrary,
there exists an automorphism $\tau_n$ of $\mathbb{N}$ such that
$x^*(i) = |x\circ\tau_n(i)|$ for all $i=1,\dots,n$. Then $(v\circ
\tau_n)^* = v^* = w$, and so setting a finite sequence $v_n =
(v\circ \tau_n (i))_{i=1}^n$, we have that its rearrangement
$v_n^*(i) \le w(i)$ for all $i=1,\dots,n$. Assume $v_n^*(i) > 0$ for all $i\in \{1,\dots,n\}$. Then applying the fact that
the function $t\mapsto \varphi(a/t)t$ is decreasing for any fixed
$a>0$, and Proposition \ref{prop:Finite}, we get the following
inequalities
\begin{align*}
\sum_{i=1}^n \varphi\left(\frac{x^*(i)}{w(i)}\right) w(i) & \le
\sum_{i=1}^n \varphi\left(\frac{x^*(i)}{v^*_n(i)}\right)v_n^*(i)
\le \sum_{i=1}^n \varphi\left(\frac{x^*(i)}{v\circ \tau_n(i)}\right)v\circ\tau_n(i)\\
& = \sum_{i=1}^n \varphi\left(\frac{|x\circ\tau_n(i)|}{v\circ \tau_n(i)}\right)v\circ\tau_n(i)
\le \sum_{i=1}^\infty \varphi\left(\frac{|x(i)|}{v(i)}\right) v(i).
\end{align*}
Now let $v_n^*(i) = 0$ for some $i\in \{1,\dots,n\}$. Then if $x^*(i)=0$, by our convention each $i$-th term of the first four sums above is equal to zero, so the inequalities hold. If $x^*(i) > 0$ then $\varphi(x^*(i)/v_n^*(i))v_n^*(i) = \infty$, and so the second sum is also infinity. By definition of $\tau_n$ there is $k\in \{1,\dots,n\}$ such that $|x\circ \tau_n(k)|>0$ and $v\circ \tau_n(k)=0$, so at least one of the term in the third sum is infinity. Finally the last sum is infinity since for $j = \tau_n(k)$ we have $|x(j)| > 0$ and $v(j) =0$. Thus in all cases the inequalities are satisfied.

Letting then $n\rightarrow\infty$ we obtain that
\[
\sum_{i=1}^\infty \varphi\left(\frac{x^*(i)}{w(i)}\right) w(i) \le
\sum_{i=1}^\infty \varphi\left(\frac{|x(i)|}{v(i)}\right) v(i),
\]
for every $v\ge0$ and $v\sim w$. This allows us to pass to the
infimum on the right side and the proof is completed.

\end{proof}

\begin{theorem}\label{th:infinite:equality1} Let $w=\{w(n)\}$ be
a decreasing weight sequence. Then for any $x=\{x(n)\}$,
\begin{align*}
\sum_{n=1}^\infty \varphi\left(\frac{x^*(n)}{w(n)}\right) w(n)& =
\inf\left\{\sum_{n=1}^\infty \varphi\left(\frac{|x(n)|}{v(n)}\right)
v(n): v\sim w, v\ge 0\right\}.
\end{align*}
\end{theorem}

\begin{proof}
As in the proof of Lemma \ref{l:Infinite:ineq} we can assume that $w$ is positive and $x(n)\to 0$ as $n\to\infty$.  If $\supp x = \mathbb{N}$ or $|\supp x| < \infty$ then there exists a one-to-one and onto mapping $\sigma:\mathbb{N} \to \mathbb{N}$ such that $|x(n)| = x^*\circ\sigma(n)$, $n\in \mathbb{N}$. If $|\supp x| = \infty$ but $\supp x \ne \mathbb{N}$, then we find $\sigma: \supp x \to \mathbb{N}$, one-to-one, onto and such that $|x(n)| = x^*\circ \sigma(n)$, $n\in \supp x$. In both cases
\[
m(x) = \sum_{n\in \supp x} \varphi\left(\frac{|x(n)|}{w\circ\sigma(n)}\right) w\circ\sigma(n).
\]
In the first case defining $v = w\circ \sigma$, we have $v\sim w$ and $v>0$.
In the second case define $v= \{v(n)\}$ such that $v(n) = w\circ \sigma(n)$ for $n\in\supp x$, and $v(n) = 0$ for $n\notin \supp x$. Since the range of $\sigma$ is $\mathbb{N}$, it is clear that $v\sim w$, $v\ge 0$, and by our convention $m(x) = \sum_{n=1}^\infty \varphi\left(\frac{|x(n)|}{v(n)}\right)v(n)$, which completes the proof.

\end{proof}

\begin{theorem}\label{integral:ineq} Let $w$ be a decreasing weight function on $I$.

\rm(i) For every $f\in L^0$ we have
\begin{equation*}
\int_I \varphi\left(\frac{f^*}{w}\right) w
= \inf \left\{\int_I \varphi\left(\frac{|f|}{v}\right)v: {v\sim w}, v\ge 0 \right\}.
\end{equation*}

\rm(ii) If $I=(0,a)$ with $a<\infty$  we also have
\begin{align*}
\int_I \varphi\left(\frac{f^*}{w}\right) w
&= \inf \left\{\int_I \varphi\left(\frac{|f|}{v}\right)v: {v\sim w}, v > 0 \right\}\\
&= \inf \left\{\int_I \varphi\left(\frac{|f|}{w\circ\sigma}\right)w\circ\sigma: \sigma\ \text{measure
 preserving mapping from $I$ onto $I$}\right\}.
\end{align*}
\end{theorem}
\begin{proof}
Similarly as in the proof of Lemma \ref{l:Infinite:ineq} we can assume that $w>0$ on $I$. Assume that the left-hand side of the equality in (i) is finite. Then reasoning similarly as in the proof of Lemma \ref{l:Infinite:ineq}, $\lim_{t\to\infty}f^*(t)=0$ and so by Ryff's Theorem
\cite[Cor. 7.6 in Chap. 2]{BS}  there exists  a measure preserving and onto map $\tau:\supp f\to \supp f^*$
such that $|f|=f^*\circ\tau$ on $\supp f$.
 If $\supp f = I$ or $|\supp f| < \infty$ then $\tau$ is, resp. can be extended to a measure preserving mapping from $I$ onto $I$. In the other case we have $\supp f^* =I$. Thus the range of $\tau$ is $I$ in both cases.
 Hence letting $v(t) = w \circ\tau(t)$ for $t\in\supp f$ and $v(t) =0$ for $t\notin \supp f$, we have $v\sim w$, $v\ge 0$, and by our convention
\[
\int_I \varphi\left(\frac{f^*}{w}\right)\,w =\int_{\supp f} \varphi\left(\frac{f^*\circ\tau}{w\circ
\tau}\right)\,w\circ\tau = \int_{\supp f} \varphi\left(\frac{|f|}{w\circ\tau}\right)\,w\circ\tau =
\int_I \varphi\left(\frac{|f|}{v}\right)\,v.
\]
Thus we get that the left-hand side of  the equality in (a) is greater than the right-hand side.

For case (ii) observe that if $a<\infty$ then we can find a measure preserving transformation $\tau$ on $I$ such that $|f|= f^*\circ \tau$. Then $v= w \circ \tau > 0$ and $v\sim w$. Thus
\begin{align*}
\int_I \varphi\left(\frac{f^*}{w}\right) w &= \int_I \varphi\left(\frac{f^*\circ\tau}{w\circ\tau}\right) w\circ \tau =
\int_I \varphi\left(\frac{|f|}{w\circ\tau}\right) w\circ \tau\\
&\ge \inf \left\{\int_I \varphi\left(\frac{|f|}{w\circ\sigma}\right)w\circ\sigma: \sigma\ \text{measure preserving mapping from $I$ to $I$}\right\}\\
 &\ge \inf \left\{\int_I \varphi\left(\frac{|f|}{v}\right)v: {v\sim w}, v > 0 \right\}.
 \end{align*}
As for the converse direction we need to prove that for every weight $v\sim w$ and $v\ge 0$, we have

\begin{equation}\label{eq:22b}
\int_I \varphi\left(\frac{f^*}{w}\right) w \le\int_I \varphi\left( \frac{|f|}{v}\right)v.
\end{equation}

If there exists $A\subset I$ with $|A| > 0$ such that $f(t) \neq 0$ and $v(t) = 0$ a.e. on $A$, then
(\ref{eq:22b}) is satisfied since its right side is equal to infinity. Thus assume further that $\supp f \subset \supp v$.
By standard approximation argument we may assume that $f$ is a simple function, and
consequently that $f^*$ is a decreasing step function as below
\[
f^*=\sum_{i=1}^N x_i^*\chi_{A_i},\qquad |f|=\sum_{i=1}^N x_i^*\chi_{E_i}
\]
with $x_i^*>0$, $A_i=[a_{i-1},a_i)$, $0=a_0<a_1<\dots <a_N<\infty$ and  $|E_i|=|A_i|=a_i-a_{i-1}$, $i=1,\dots,N$.
Let $A=[0,a_N)$, $E=\bigcup E_i$,  $v_E=v\chi_E$ and $v_E^* = (v_E)^*$. By the assumption that $E=\supp f \subset \supp v$ we have that $v_E > 0$ on $E$ and $v_E^* > 0$ on $A$. Moreover
 $v_E^*   \le v^*=w^*$.  Hence
\[
\int_I \varphi\left(\frac{f^*}{w}\right) w = \int_A \varphi\left(\frac{f^*}{w}\right) w
\le\int_A \varphi\left( \frac{f^*}{v^*_E}\right)v^*_E
=\int_I \varphi\left( \frac{f^*}{v^*_E}\right)v^*_E.
\]
For each $i$ there is a measure-preserving transformation $\tau_i:A_i\to E_i$. Let
$\tau: A\to E$ be the measure preserving transformation defined by $\tau(t)=\tau_i(t)$
whenever $t\in A_i$. Then $|f|\circ\tau= f^*$, $v_E\circ\tau$ is a positive function on $A$ with
$(v_E\circ\tau)^*=v_E^*$ and
\[
\int_I \varphi\left( \frac{|f|}{v}\right)v=\int_A \varphi\left(
\frac{f^*}{v_E\circ\tau}\right)\,v_E\circ\tau.
\]
Thus it is sufficient to prove (\ref{eq:22b}) with $v_E^*,
v_E\circ\tau,f^*,A$ respectively in place of $w,v,f,I$. Observe that both $v_E^*$ and $v_E\circ \tau$ are positive on $A$.
Therefore we can reduce the proof to the case where $I$ has finite measure, $f$ is a
decreasing, positive simple function and $v$ is positive on $I$. Now, another approximation argument
allows to reduce to the case where  $w$ is also a simple function
(approximating e.g. $w$ by $w+\eps$, $\eps>0$,  and then $w+\eps$
from below by simple decreasing functions $w_{n,\eps}\ge\eps$).
Note that if $v\sim w$ then $v=w\circ\sigma$ for some measure
preserving map $I\to I$, and $v$ is approximated by the simple
functions $w_{n,\eps}\circ\sigma$. Also, refining the partition if
necessary, we may suppose that the function $w$ is built on the
same intervals as the function $f^*=f$. Thus we may suppose
$w=\sum_{i=1}^N w_i\chi_{A_i}$ and $v=\sum_{i=1}^Nw_i\chi_{B_i}$
with $|B_i|=|A_i|$ and $w_i > 0$. Let $a_{ij}=|A_i\cap B_j|$. Since $(A_i)$ and
$(B_i)$ are two partitions of the interval $I$ we have $\sum_j
a_{ij}=|A_i|$, $\sum_i a_{ij}=|B_j|=|A_j|$, so $\sum_j
a_{ij}=\sum_j a_{ji}$ for all $i=1,\dots,N$. Finally applying  Corollary
\ref{cor:finite} we get
\[
\int_I \varphi\left(\frac{f^*}{w}\right)\,w
=\sum_{i,j=1}^Na_{ij}\varphi\left(\frac{x_i}{w_i}\right)\,w_i
\le \sum_{i,j=1}^Na_{ij}\varphi\left(\frac{x_i}{w_j}\right)\,w_j
=\int_I\varphi\left(\frac{f}{v}\right)\,v,
\]
and the proof is completed.
\end{proof}

If we wish to approximate $m(x)$ in Theorem \ref{th:infinite:equality1}, or $M(f)$ in Theorem \ref{integral:ineq} for $a=\infty$, by the infimum  over positive  $v$ or $w\circ\sigma$ where $\sigma$ is an automorphism of $\mathbb{N}$ or measure preserving transformation of  $(0,\infty)$, we need some additional assumptions on $w$ and $\varphi$ as we will see  below.
We start with a preparatory lemma.

\begin{lemma}\label{lem:2} Let $w$ be a positive decreasing weight function on $(0,\infty)$. Let $f\in L^0$ be such that
$lim_{t\to\infty}f^*(t) = 0$,  $|\supp f| = \infty$ and $|(\supp f)^c| > 0$. Then for every $T>0$ there exists a  measure preserving and surjective mapping $\sigma: (0,\infty)\to (0,\infty)$ such that
\[
\int_0^\infty \varphi\left(\frac{|f|}{w\circ\sigma}\right) w\circ\sigma \le \int_0^T \varphi\left(\frac{f^*}{w}\right) w + \int_T^\infty \varphi\left(\frac{f^*(t)}{w(2t)}\right) w(2t)\,dt.
\]
An analogous statement remains true in discrete case, where $f$ is replaced by a sequence,  $\sigma$ by a one-to-one and onto mapping of $\mathbb{N}$, and integration by summation.

\end{lemma}

\begin{proof}
We shall prove this only in the case when $f\in L^0$ is such that $\lim_{t\to\infty}f^*(t) = 0$ and $|\supp f| = |(\supp f)^c| = \infty$. Then there exist $\tau: \supp f \to \supp f^* = (0,\infty)$ and $\rho:  (\supp f)^c\to (0,\infty)$ such that both are measure preserving and onto, and moreover $|f|=f^*\circ \tau$. Let $T>0$. Define
   \[
  \sigma(t)=  \begin{cases}
     &  \tau(t)+jT\ \  \text{if}\ \ t\in \tau^{-1}(jT,(j+1)T) , j=0,1,\dots; \\
     &  \rho(t)+(j-1)T\ \  \text{if}\ \ t\in  \rho^{-1}(jT,(j+1)T) , j=1,2,\dots.
  \end{cases}
  \]
  It is easy to check that $\sigma:(0,\infty)\to (0,\infty)$ is measure preserving and onto.
Then
  \begin{align*}
\int_0^\infty &\varphi\left(\frac{|f|}{w\circ\sigma}\right) w\circ\sigma = \int_{\supp f} \varphi\left(\frac{|f|}{w\circ\sigma}\right) w\circ\sigma
=\int_{\supp f} \varphi\left(\frac{f^*\circ\tau}{w\circ\sigma}\right) w\circ\sigma\\
= &\int_{\tau^{-1}(0,T)}\varphi\left(\frac{f^*\circ\tau}{w\circ\tau}\right) w\circ\tau + \sum_{j=1}^\infty
\int_{\tau^{-1}(jT,(j+1)T)}\varphi\left(\frac{f^*(\tau(t))}{w(\tau(t)+jT)}\right)w(\tau(t)+jT)\,dt \\
=& \int_0^T \varphi\left(\frac{f^*}{w}\right) w +  \sum_{j=1}^\infty \int_{jT}^{(j+1)T} \varphi\left(\frac{f^*(u)}{w(u+jT)}\right)
w(u+jT)\,du.
\end{align*}

Now by decreasing monotonicity of the function $t\mapsto \varphi(c/t)t$ for any $c>0$, we have that for $u\in (jT,(j+1)T)$
it holds
\[
\varphi\left(\frac{f^*(u)}{w(u+jT)}\right) w(u+jT) \le \varphi\left(\frac{f^*(u)}{w(2u)}\right) w(2u).
\]
Hence
\begin{align*}
\int_0^\infty \varphi\left(\frac{|f|}{w\circ\rho}\right) w\circ\rho
&  \le \int_0^T \varphi\left(\frac{f^*}{w}\right) w +
\sum_{j=1}^\infty \int_{jT}^{(j+1)T} \varphi\left(\frac{f^*(u)}{w(2u)}\right) w(2u)\,du \\
& = \int_0^T \varphi\left(\frac{f^*}{w}\right) w + \int_T^\infty \varphi\left(\frac{f^*(u}{w(2u)}\right) w(2u)\, du,
\end{align*}
and the proof is completed.

\end{proof}

Now we are ready to present more refined results than Theorems \ref{th:infinite:equality1} and \ref{integral:ineq}, but with an additional assumption of control of the decreasing slope of $w$ in relation to $\varphi$.

\begin{definition} A positive weight function (resp., sequence) $w$ is said to be controlled by the Orlicz function $\varphi$ (or shortly ``to be {\it $\varphi$-controlled}\,'') if for some $K>0$ and every $c>0$, $t\in I$ with $2t\in I$, (resp., every $n\in \mathbb{N}$) it holds
\[
\varphi\left(\frac c {w(2t)}\right)w(2t)\le K\varphi\left(\frac c {w(t)}\right)w(t)
\ \ \left(\text{resp.,}\ \ \ \varphi\left(\frac c {w(2n)}\right)w(2n)\le K\varphi\left(\frac c {w(n)}\right)w(n)\right).
\]
\end{definition}
\begin{remark} (i) Trivial cases are the constant weights that are controlled by every Orlicz function, and the Orlicz function $\varphi(t)=t$ which controls every weight.
\par (ii) If $\varphi$ and $1/w$ satisfy $\Delta_2$ condition, then $w$ is $\varphi$-controlled. Notice that if $w$ is regular then $1/w$ satisfies $\Delta_2$-condition. In fact, by regularity of $w$ for some $C>0$ and all $t\in I$, $W(t) \le C tw(t)$. Hence $tw(t) \le W(2t) \le 2Ctw(2t)$ and so $1/w(2t) \le 2C/w(t)$ which means that $1/w$ satisfies $\Delta_2$ condition.
\end{remark}

Recall that the dilation operator $D_2$ is defined for $f\in L^0$ as $D_2f(t) = f(t/2)$, $t\in I$, and for a sequence $x$ as $D_2 x (n) = x(\lceil n/2\rceil)$.

\begin{theorem}\label{integral:ineq2} Let $\varphi$ be an Orlicz function. $\rm(i)$ Let $w$ be a $\varphi$-controlled decreasing  weight function on $I = (0,\infty)$.   Then for every $f\in L^0$,
\begin{align*}
\int_0^\infty \varphi\left(\frac{f^*}{w}\right) w
&= \inf \left\{\int_0^\infty \varphi\left(\frac{|f|}{v}\right)v: {v\sim w}, v > 0 \right\}\\
&= \inf \left\{\int_0^\infty \varphi\left(\frac{|f|}{w\circ\sigma}\right)w\circ\sigma: \sigma\ \text{ measure preserving mapping from $I$ to $I$}\right\}.
\end{align*}

$\rm(ii)$ Let $w=\{w(n)\}$ be a $\varphi$-controlled decreasing weight sequence.  Then for any $x=\{x(n)\}$,
\begin{align*}
\sum_{n=1}^\infty \varphi\left(\frac{x^*(n)}{w(n)}\right) w(n)& =
\inf\left\{\sum_{n=1}^\infty \varphi\left(\frac{|x(n)|}{v(n)}\right)
v(n): v\sim w, v>0\right\}\\
& = \inf\left\{\sum_{n=1}^\infty \varphi\left(\frac{|x(n)|}{w\circ
\sigma(n)}\right) w\circ \sigma(n): \sigma \ \text{automorphism of $\mathbb{N}$}\right\}.
\end{align*}
If $w$ is not necessarily $\varphi$-controlled then the above
equalities hold true for those functions $f$ or sequences $x$ for which $M(D_2f)< \infty$ or $m(D_2x)< \infty$, respectively.

\end{theorem}

\begin{proof}
 We shall only prove $\rm(i)$. By the proof of Theorem \ref{integral:ineq} the left-hand side of the first equality is less than its right-hand side, which in turn is less than the right-hand side of the second equality. For a converse it is enough to show that for any $\epsilon > 0$ there is a measure preserving and onto mapping $\sigma$ of $(0,\infty)$ such that $\int_0^\infty \varphi\left(\frac{|f|}{w\circ\sigma}\right) w\circ\sigma \le M(f) + \epsilon$. Assume that $f\in L^0$ is such that
\[
M(f) = \int_0^\infty \varphi\left(\frac{f^*}{w}\right)w < \infty.
\]
It follows that $\lim_{t\to\infty} f^*(t) = 0$. If $|\supp f| < \infty$ or $\supp f = I$, then there exists a measure preserving transformation $\tau$ of $I$ such that $|f| = f^* \circ \tau$. Hence
$M(f) = \int_0^\infty \varphi\left(\frac{|f|}{w\circ\tau}\right) w\circ\tau$, and the equalities hold.

Suppose now that $|\supp f| = \infty$ and $|(\supp f)^c| > 0$. By the assumption that $w$ is  $\varphi$-controlled,
 there is $C>0$ such that  $\varphi\left(\frac \alpha{w(2t)}\right) w(2t)\le C\varphi\left(\frac \alpha{w(t)}\right) w(t)$ for every $\alpha,t>0$. So
\[
\int_0^\infty \varphi\left(\frac{f^*(t)}{w(2t)}\right) w(2t)\, dt \le  CM(f) < \infty.
\]
Then for any $\epsilon > 0$ there exists $T>0$ such that
\[
\int_T^\infty \varphi\left(\frac{f^*(t)}{w(2t)}\right) w(2t)\, dt <  \epsilon.
\]
An appeal to Lemma \ref{lem:2} gives now a suitable measure preserving mapping.

\end{proof}

The next result and its corollary provide another useful reformulations of $m(x)$ or $M(f)$.

\begin{proposition}

Let $w$ be a weight function. Then for every $f\in L^0$,
\begin{equation*}\label{eq:222}
\inf\left\{ \int_I \varphi\left(\frac{|f|}{v}\right) v: v^*\le w, v>0\right\} = \inf\left\{\int_I \varphi\left(
\frac{|f|}{v}\right)v: v\sim w, v>0\right\}.
\end{equation*}
Let $w=\{w(n)\}$ be a weight sequence. Then for
every sequence $x=\{x(n)\}$,
\begin{equation*}\label{eq:111}
\inf \left\{\sum_{n=1}^\infty \varphi\left(\frac{|x(n)|}{v(n)}\right) v(n):  v^*\le w, v>0\right\} =
\inf\left\{\sum_{n=1}^\infty \varphi\left(\frac{|x(n)|}{v(n)}\right) v(n):  v\sim w, v>0\right\}.
\end{equation*}
In the above formulas we can replace simultaneously in both sides  $v>0$ by $v\ge 0$.
\end{proposition}

\begin{proof} We shall prove it only for the function case, assuming $I=(0,\infty)$. Since for
every $v\sim w$, we have $v^*=w$, the left side is less than the right side. In
the opposite direction we assume that the left side is finite. Consider any measurable positive
function $k:I\to\mathbb (0,\infty)$ with $k^*\le w$ and $\int\varphi(|f|/k)k<\infty$.

Let $\alpha=\lim_{t\to\infty}k^*(t)$ and  $A= \{t: k(t) > \alpha\}$. Consider two cases.

Let first $|A| = \infty$. Then $((k-\alpha)\chi_A)^* = k^* - \alpha$ on $I$, and $\supp (k-\alpha)\chi_A = A$. By Ryff's theorem there exists an onto measure preserving transformation $\tau:A\to I$ such that for $t\in A$, $(k-\alpha)\chi_A(t) = (k-\alpha)^*\circ \tau(t)$. Hence for $t\in A$, $k(t) - \alpha = k^*\circ\tau(t) - \alpha$, and so $k(t) = k^*\circ\tau(t)$ on $A$.
Define $v(t) = w\circ\tau(t)$ if $t\in A$, and $v(t) = k(t)$ if $t\notin A$. Since $k^* \le w$, so for $t\in A$, $k(t) = k^*\circ\tau(t) \le w\circ(t) = v(t)$. Thus $k\le v$ on $I$. Moreover, since for $t\in A$, $v(t) = w\circ\tau(t) \ge k(t) > \alpha$, and for $t\notin A$, $v(t) \le \alpha$, in view of $|A| = \infty$, we have that for $t\in I$, $v^*(t) = (v\chi_A)^*(t) = ((w \circ\tau)\chi_A)^*(t) = w(t)$.

Let now $|A| < \infty$. Then let $B= \{t:  k(t) = \alpha\}$. Since $\lim_{t\to\infty} k^*(t) = \alpha$, we have $|B| = \infty$.
Then there exists $\tau:A \to \{t: k^*(t)>\alpha\}$ and $\tau:B\to \{t: k^* = \alpha\}$, such that $\tau:A\cup B \to I$ is onto and measure preserving. Moreover $k(t) = k^*\circ\tau(t)$ for $t\in A\cup B$. Let now $v(t) = w\circ\tau(t)$ for $t\in A\cup B$, and $v(t) = k(t)$ otherwise. Clearly $v\ge k$ on $I$. In view of $v(t) \ge \alpha$ on $A\cup B$ and $|A\cup B| = \infty$, we have that $(v\chi_{A\cup B})^* = v^*$. Hence $v^* = ((w\circ\tau)\chi_{A\cup B})^* = w$.

In both cases we have found $v$ such that  $v^*=w$ and $v\ge k$. Hence
\[
\int_I \varphi\left( \frac{|f|}{v}\right)v\le \int_I \varphi\left( \frac{|f|}{k}\right)k.
\]
We complete the proof taking infimum first with respect to $v\sim w$ and then $k$ such that  $k^*\le w$.
\end{proof}

\begin{corollary}\label{cor:2} For any function $f\in L^0$,
\[
M(f) = \int_I \varphi\left(\frac{f^*}{w}\right) w = \inf \left\{ \int_I \varphi\left(\frac{|f|}{v}\right) v: v^* \le w, v \ge 0 \right\},
\]
and for any sequence $x= \{x(n)\}$,
\[
m(x) = \sum_{n=1}^\infty \varphi\left(\frac{x^*(n)}{w(n)}\right) w(n) = \inf \left\{ \sum_{n=1}^\infty \varphi\left(\frac{|x(n)|}{v(n)}\right): v^* \le w, v\ge 0\right\}.
\]

\end{corollary}

\section{The classes $M_{\varphi,w}$}\label{sec:Classes M}

In this section we investigate several aspects of the classes $M_{\varphi,w}$, which were defined in section \ref{prelim}: linear structure, concavity and K\"othe duality.

\subsection{Linear structure}

It is known  from the general theory  \cite[Lemma 1.4]{KR} that $M_{\varphi,w}$, resp. $m_{\varphi,w}$, is a linear space and the corresponding homogeneous functional $ \|\cdot \|_M$, resp. $ \|\cdot \|_m$, is a quasi-norm if and only if the dilation operator $D_2$  on the class $M_{\varphi,w}$, resp. $m_{\varphi,w}$, is bounded.  In this case $M_{\varphi,w}$, resp. $m_{\varphi,w}$ are complete, and thus quasi-Banach spaces.
 In view of \cite[Proposition 4.5]{KR} a simple sufficient condition for boundedness of $D_2$ on $M_{\varphi,w}$ is the following inequality
\begin{equation}\label{doubling condition}
 \varphi\left(\frac c {w(2t)}\right)w(2t)\le  \varphi\left(\frac {C c} {w(t)}\right)w(t)
\end{equation}
for some $C>0$ and all $c >0$, $t>0$  with $2t\in I$.
By the fact that $t\mapsto \varphi(c/t) t$ is decreasing for $t>0$ and any $c>0$, it holds in particular if $1/w$ verifies condition $\Delta_2$ (regardless of $\varphi$). Recall that any function $h:I \to \mathbb{R}_+$ satisfies condition $\Delta_2$ if $h(2t) \le K h(t)$ for some $K>0$ and all $t\in I$ such that $2t \in I$.

The class $M_{\varphi,w}$ cannot be linear when $D_2$ is not bounded on it. More precisely we have the following result.

\begin{lemma}\label{lem:quasi-norm}
The following assertions are equivalent.
\begin{itemize}
\item[(\rm i)] The dilation operator $D_2$ acts  on the class $M_{\varphi,w}$.
\item[(\rm ii)] The dilation operator $D_2$ acts  and is bounded on the class $M_{\varphi,w}$.
\item[(\rm iii)] The class $M_{\varphi,w}$ is linear.
\item[(\rm iv)]  The class $M_{\varphi,w}$ is linear and $\|\cdot \|_M$ is a quasi-norm.
\end{itemize}
\end{lemma}
A trivial, but useful fact which will be used in the proof of this lemma is that the cones of decreasing, nonnegative functions in $M_{\varphi,w}$ and in the Banach function space $\mathcal L_{\varphi,w}$ coincide, where
$\mathcal L_{\varphi,w}=\{f\in L^0: \|f\|_{\mathcal L_{\varphi,w}}: = \inf\{\epsilon > 0: \int_I \varphi(|f|/(\epsilon w))w \le 1\}< \infty \}$,
  and the homogeneous functional $\|\cdot\|_M$ coincides with the norm $\|\cdot\|_{\mathcal L_{\varphi,w}}$ on this cone. In particular this cone is $\sigma$-convex that is closed under infinite convex combinations, and the restriction of the functional $\|\cdot\|_M$ to this cone is $\sigma$-convex.
\begin{proof}[Proof of Lemma \ref{lem:quasi-norm}]
(i) $\implies$ (iii) and (ii) $\implies$ (iv)  follow from the classical inequality $(f+g)^*\le D_2f^*+D_2g^*$.
(iii) $\implies $ (i) and (iv) $\implies$ (ii) follow from the equality $(f_1+f_2)^*=D_2f^*$ whenever $f_1$ and $f_2$ are two disjoint  functions both of which are equimeasurable with $f$.

Clearly (ii) implies (i). Assume that (i) holds true but not (ii). There exists a sequence $(f_n)$ in $M_{\varphi,w}$ with $\|D_2 f_n\|_M> 4^n \|f_n\|_M$ and $\|f_n\|_M \le 1$ for all $n\in\mathbb{N}$. We may assume w.l.o.g. that $f_n=f_n^*$ and $\|f_n\|_M=1$ for every $n\ge 1$. Set $f=\sum_{n=1}^\infty 2^{-n}f_n$. Since the cone of positive decreasing functions in $M_{\varphi,w}$ is $\sigma$-convex, it holds that $f\in M_{\varphi,w}$. Note that  $D_2f\ge 2^{-n}D_2f_n$ for every $n\ge 1$. Thus if $D_2f\in M_{\varphi,w}$ we obtain $\|D_2f\|_M\ge 2^{-n}\|D_2 f_n\|_M \ge 2^{-n}\times 4^n=2^n$ for every $n\ge 1$, a contradiction. Hence (i) implies (ii).
\end{proof}

Denote by $F_M$ the fundamental function of $M_{\varphi,w}$, defined as usual by
\[
F_M(t)=\|\chi_{(0,t)}\|_M, \ \ \ t\in I.
\]
Note that for every $t > 0$ such that $2t\in I$, $D_2\chi_{[0,t]}= \chi_{[0,2t]}$. Thus $D_2$ is bounded on characteristic functions if and only if the fundamental function $F_M$ verifies the condition $\Delta_2$. In particular if $F_M$ does not verify the condition $\Delta_2$ then the class $M_{\varphi, w}$ is not linear. Below we will give a partial converse of that statement.
But let us first give a simple explicit criterion for the fundamental function $F_M$ to verify condition $\Delta_2$. For $t\in I$ we set
\[
G_M(t)=  {1\over w(t)\varphi^{-1}\left({1\over tw(t)}\right)}.
\]
Note that the function $G_M$ is increasing. Indeed if $0<s\le t\in I$ then
\[w(t)\varphi^{-1}\left({1\over tw(t)}\right)\le w(t)\varphi^{-1}\left({1\over sw(t)}\right)\le w(s)\varphi^{-1}\left({1\over sw(s)}\right)\]
since the function $u\mapsto\varphi^{-1}(cu)/u$ is decreasing on $(0,\infty)$ for any $c > 0$.

\begin{lemma}\label{Fundamental M}
The  fundamental function $F_M$ verifies the condition $\Delta_2$ if and only if the function $G_M$ does. Moreover if this is the case, then the functions $F_M$ and $G_M$ are equivalent.
\end{lemma}
\begin{proof}
 Indeed, since for every $c>0$ the function $s\mapsto \varphi\left(\frac c{w(s)}\right)w(s)$ is increasing, we have for $t\in I$,
\[1=\int_0^t \varphi \left({1\over F_M(t)w(s)}\right)w(s)ds\le  t\varphi \left({1\over F_M(t)w(t)}\right)w(t),\]
and similarly
\[1=\int_0^{2t }\varphi \left({1\over F_M(2t)w(s)}\right)w(s)ds\ge \int_t^{2t }\varphi \left({1\over F_M(2t)w(s)}\right)w(s)ds\ge t\varphi \left({1\over F_M(2t)w(t)}\right)w(t).\]
Hence for $t>0$ such that $2t\in I$,
\begin{equation}\label{estim F_M}
F_M(t)\le G_M(t) \le F_M(2t).
\end{equation}
Thus if $F_M$ verifies condition $\Delta_2$, then $F_M$ and $G_M$ are equivalent and $G_M$ must also satisfy condition $\Delta_2$. Moreover rewriting (\ref{estim F_M}) as
\[G_M\left(t/2\right)\le F_M(t)\le G_M(t)\]
we see that if $G_M$ verifies $\Delta_2$ then $F_M$ and $G_M$ are equivalent and $F_M$ verifies also condition $\Delta_2$.
\end{proof}

Recall now the lower and upper (Matuszewska-Orlicz) indices \cite{MO, BS, LT2} for a function $h: I\to \mathbb{R_+}$,
\begin{align*}
\alpha_h &= \sup\{p\in \mathbb{R}: \exists C > 0\ \forall t\in I\ \forall\, 0<\lambda \le 1\ \ h(\lambda t) \le C \lambda^p h(t) \},
\\
\beta_h &= \inf\{p\in \mathbb{R}: \exists C > 0\ \forall t\in I\ \forall 0<\lambda \le 1\ \ h(\lambda t) \ge C \lambda^p h(t)\}.
\end{align*}
Clearly the indices are preserved by equivalent functions. Notice that for Orlicz function $\varphi$ we have  $I=\Bbb{R}_+$ in the definitions of indices.

\begin{lemma}
Assume that $\alpha_\varphi > 1$. Then the fundamental function of $M_{\varphi,w}$ satisfies condition $\Delta_2$ if and only if $1/w$ does.
\end{lemma}
\begin{proof}
In view of Lemma \ref{Fundamental M} we have only to prove that $G_M$ verifies $\Delta_2$  if and only $1/w$ does.

If $1/w$ satisfies $\Delta_2$, that is $C=\sup\limits_{t\in I}{ w(t)\over w(2t)}<\infty$ then using the fact that $w$ is decreasing and $\varphi^{-1}$ is concave we get for $t\in I$,
\[G_M(2t)= {\frac 1{w(2t)}\over \varphi^{-1}(\frac1{2t w(2t)})} \le {\frac C{w(t)}\over \varphi^{-1}(\frac1{2t w(t)})}\le   {\frac {C}{w(t)}\over \frac 12\varphi^{-1}(\frac1{t w(t)})}= 2CG_M(t),\]
and so $G_M$ fulfils  $\Delta_2$ condition.

Assume now that $G_M$ verifies the condition $\Delta_2$ and set $C=\sup\limits_{t\in I}{G_M(2t)\over G_M(t)} < \infty$. By hypothesis we have $\alpha_\varphi>1$, which implies that $\beta_{\varphi^{-1}} = 1/\alpha_\varphi<1$, that is for some $p<1$, $d\ge1$,  and all $u>0$, $\lambda\ge 1$   we have $\varphi^{-1}(\lambda u)\le d\lambda^p\varphi^{-1}(u)$ . Then for $2t\in I$ we have
\[ CG_M(t)\ge G_M(2t)\ge {\frac 1{w(2t)}\over d \left({w(t)\over w(2t)}\right)^p \varphi^{-1}\left(\frac 1{2tw(t)}\right)} = \frac 1d\left({w(t)\over w(2t)}\right)^{1-p}{\frac 1{w(t)}\over \varphi^{-1}\left(\frac 1{2tw(t)}\right)}\ge  \frac 1d\left({w(t)\over w(2t)}\right)^{1-p}G_M(t).
\]
Hence
$ {w(t)/ w(2t)}\le (dC)^{1\over 1-p}$ that is $1/w$ satisfies condition $\Delta_2$.
\end{proof}

Summing up the preceding results, we obtain a nice characterization of the normability for $M_{\varphi,w}$ classes, at least when $\varphi$ has  convexity ``better than 1''.

\begin{proposition}\label{prop:equiv linearity}
Assume that $\alpha_\varphi > 1$. Then the following assertions are equivalent.
\begin{itemize}
\item[(\rm i)] $M_{\varphi,w}$ is a linear space.
\item[(\rm ii)] $M_{\varphi,w}$ is a linear space, and $\|\cdot \|_M$ is a quasinorm.
\item[(\rm iii)] The fundamental function of  $M_{\varphi,w}$ satisfies condition $\Delta_2$.
\item[(\rm iii)] $1/w$ satisfies condition $\Delta_2$.
\end{itemize}
\end{proposition}

\subsection{Concavity}

We draw now some consequences of section \ref{sec:New formulas} on the geometry of spaces $M_{\varphi,w}$.

\noindent We say that an Orlicz function $\varphi$ is $p$-concave for some $1<p<\infty$ if the map  $t\mapsto \varphi(t^{1/p})$ on $\mathbb{R}_+$ is concave. Similarly a modular $G: L^0\to \mathbb R_+$ is $p$-concave if the map $f\mapsto G(f^{1/p})$ is concave on $L_0^+(I)$.
\begin{proposition}\label{prop:concavity}

\noindent \rm{(1)} The modular $M$ is disjointly superadditive, that is $M(f+g)\ge M(f)+ M(g)$ whenever $f,g\in L^0$ are disjoint.

\noindent \rm{(2)} If  the Orlicz function $\varphi$ is $p$-concave for some $1<p<\infty$ then so is the modular $M$.
\end{proposition}

\begin{proof}
(1) results from the fact that by Theorem \ref{integral:ineq}  the modular $M$ is the infimum of the disjointly additive maps $\displaystyle I_v: f\mapsto \int_I\varphi\left(|f|\over v\right)v$, where $v\ge 0, v\sim w$.

 Similarly (2) follows from the fact that the modulars $I_v$ are clearly $p$-concave whenever $\varphi$ is $p$-concave, and an infimum of $p$-concave modulars is $p$-concave.
\end{proof}

\begin{corollary}\label{cor:316}
If the Orlicz function $\varphi$ is equivalent to a $p$-concave Orlicz function for some $1<p<\infty$, then the class $M_{\varphi,w}$ is $p$-concave, that is for some constant $c>0$, for all $f_1,\dots, f_n\in M_{\varphi,w}$ and every $n\in\mathbb N$,
\[   \left \|\left(\sum_{i=1}^n|f_i|^p\right)^{1/p}\right\|_M\ge c \left(\sum_{i=1}^n\|f_i\|_M^p\right)^{1/p}.\]
\end{corollary}

\begin{proof} Let $|g|^p=\sum_{i=1}^n|f_i|^p$. It is easy to check that if Orlicz functions are equivalent then the classes $M_{\varphi_1,w}$ and $M_{\varphi_2,w}$ coincide and $\|\cdot\|_{M_{\varphi_1,w}}$ and $\|\cdot\|_{M_{\varphi_2,w}}$ are equivalent.
Thus we may assume w.l.o.g. that $\varphi$ is $p$-concave, and that $\sum_{i=1}^n\|f_i\|_M^p=1$. Let $c_i=\|f_i\|_M$ and $g_i=c_i^{-1}f_i$. We have $|g|^p=\sum_{i=1}^n c_i^p|g_i|^p$ with $\sum_{i=1}^n c_i^p=1 $. Denote $G(f)=M(|f|^p)$ and notice that $M(tg_i) > 1$ for $t>1$. Then by concavity of $G$ we get for all $t>1$,
\[ M(tg)=G(t^p|g|^p)\ge \sum_{i=1}^n c_i^pG(t^p|g_i|^p)= \sum_{i=1}^n c_i^pM(tg_i)> \sum_{i=1}^n c_i^p=1,\]
and thus $\|g\|_M\ge 1$.
\end{proof}

\subsection{K\"othe duality}

By {\it K\"othe dual} of a subset $A$ of $L^0$ we mean the set of elements $f$ of $L^0$ such that $fg\in L_1$ for any $g\in A$, where as usual $L_1$ denotes the space of integrable functions on $I$ equipped with the norm $\|f\|_1 = \int_I|f|$. We denote the K\"othe dual  of $A$ by $A'$. Note that $A'$ is a solid linear subspace in $L_0(I)$. In this subsection we shall determine the K\"othe dual $M'_{\varphi,w}$ of the class $M_{\varphi,w}$. We do not assume this class to be quasi-normed nor even linear.

\begin{lemma}\label{rearrangement-in-M'}
Every element of $M'_{\varphi,w}$ has a decreasing rearrangement which belongs also to $M'_{\varphi,w}$.
\end{lemma}

\begin{proof}
Letting $f\in M'_{\varphi,w}$,  we may assume w.l.o.g. that $f\ge 0$.

i)  We show first  that $f$ has a finite rearrangement. If it is not the case, then for all $\lambda>0$, $|\{t: f(t)>\lambda\}|=\infty$. Then we can find recursively a sequence $(A_n)_{n=1}^\infty$ of disjoint measurable sets of measure 1, such that for every  $n\in \mathbb{N}$ we have $A_n\subset \{t: f(t)>2^n F_M(n)\}$.  Then setting
\[g=\sum_{n=1}^\infty {2^{-n}\over F_M(n)}\chi_{A_n},\]
we get
\[g^*=\sum_{n=1}^\infty {2^{-n}\over F_M(n)}\chi_{[n-1,n]}\le \sum_{n=1}^\infty {2^{-n}\over F_M(n)}\chi_{[0,n]}:=h.\]
Clearly $M(h)\le \sum_{n= 1}^\infty 2^{-n}=1$, and so $M(g)=M(g^*)\le 1$. On the other hand
\[\int_I fg=\sum_{n=1}^\infty\int_{A_n}{2^{-n}\over F_M(n)}f(t)dt=\infty,\]
a contradiction.

ii)  Thus $f^*(t) < \infty$ for all $t\in I$. Set for every $n\in\mathbb Z$,
\[B_n=\{t\in I: 2^n\le f<2^{n+1}\}\  \hbox{ and } \ f_1= \sum_{n\in\mathbb Z} 2^n \chi_{B_n}.  \]
Then $f_1\le f\le 2f_1$, and it is sufficient to prove that $f_1^*\in M'_{\varphi,w}$. Since the set of values of $f_1$ has zero as a unique possible accumulation point, we have
\[
f_1^*= \sum_{n\in J} 2^n \chi_{I_n},
\]
where $J$ is an interval of $\mathbb Z$ and $I_n$ an interval of $I$ with $|I_n|=|B_n|$ for every $n\in J$.
The first element in $J$ is the supremum of those $n$ that $I_n$ has infinite measure, and the last element of $J$ is the supremum of all integers $n$ such that $I_n$ has positive finite measure. In particular if for each $n\in\mathbb{Z}$, $0<|B_n| < \infty$, then $\inf J = -\infty$ and $\sup J = \infty$, and then $f_1^* = \sum_{n=-\infty}^\infty 2^n \chi_{I_n}$.

Let $B=\bigcup\limits_{n\in J} B_n$ and $f_2=\chi_B f_1$. Then $f_1^*=f_2^*$ and there is a measure preserving transformation and onto $\sigma: B\to S$, where $S$ is the support of $f_1^*$,  such that $f_2=f_1^*\circ\sigma$. If $g\in M_{\varphi,w}$,  we have
\[\int f_1^*|g|=\int (f_1^*\circ \sigma)(|g|\circ\sigma)=\int f_2(|g|\circ\sigma).\]
This last integral is finite since $0\le f_2\le f\in M'_{\varphi,w}$ and $(g\circ\sigma)^*=(g\chi_S)^*\le g^*\in M_{\varphi,w}$. It follows that  $f_1^* \in M'_{\varphi,w}$ and the proof is completed.
\end{proof}

\begin{lemma}\label{M'-norm}
For every $f\in M'_{\varphi,w}$ it holds that $\sup\{\|fg\|_1: g\in M_{\varphi,w}, \|g\|_M\le 1\}<\infty$.
\end{lemma}
\begin{proof}
By the ordinary Hardy-Littlewood inequality and the fact that $f^*$ belongs to $M'_{\varphi,w}$ too, we may w.l.o.g. suppose that $f$ is non-negative and decreasing and show that
\[
\sup\{\|fg\|_1: g\ge 0, \hbox{ decreasing, } \|g\|_M\le 1\}<\infty.
\]
If it is not the case, we may find a sequence $(g_n)$ of decreasing functions in $M_{\varphi,w}$, with $\|g_n\|_M\le 1$ and
$\int_I fg_n\ge 2^n$, $n\in\mathbb{N}$. Setting $g=\sum_{n=1}^\infty 2^{-n}g_n$ we have that $g\ge 0$, $g$ is decreasing and
\[ \|g\|_M\le \sum_{n=1}^\infty 2^{-n}\|g_n\|_M\le 1,\]
since $\|\ \|_M$ is convex on the cone of non-negative decreasing elements of $M_{\varphi,w}$. On the other hand
\[\int_I fg=\sum_{n=1}^\infty 2^{-n}\int_I fg_n\ge \sum_{n=1}^\infty 2^{-n}\times 2^n=\infty,\]
a contradiction.
\end{proof}

Given a weight function $w$, by  $L_\varphi(w)$ denote the Orlicz space associated to the Orlicz function $\varphi$ and the measure $wdt$, that is $L_\varphi(w) = \{f\in L^0: \exists \lambda > 0 \ \int_I \varphi(\lambda |f|) w < \infty\}$. It is equipped with the Luxemburg norm $\|f\|_{L_\varphi(w)} = \inf\{\epsilon > 0: \int_I \varphi(|f|/\epsilon) w \le 1\}$. Recall also that the Orlicz-Lorentz space $\Lambda_{\varphi,w}$ consists of all $f\in L^0$ such that $f^* \in L_\varphi(w)$.

\begin{definition}\label{def}
Let $\varphi$ be  $N$-function, and $\varphi_*$ be its complementary function. For  $f\in\Lambda_{\varphi_*,w}$ we denote by  $\|f\|^0_{\Lambda_{\varphi_*,w}}$ the Orlicz norm of $f^*$ in the weighted Orlicz space $L_{\varphi_*}(w)$, that is
\[ \|f\|^0_{\Lambda_{\varphi_*,w}}= \|f^*\|^0_{L_{\varphi_*} (w)} = \sup\left\{ \int_I f^*gw: g\in L_{\varphi}(w), \|g\|_{L_{\varphi}(w)}\le 1\right\}. \]
We call $\|f\|^0_{\Lambda_{\varphi_*,w}}$ the {\it Orlicz norm} of $f$ in $\Lambda_{\varphi_*,w}$ and by $\Lambda^0_{\varphi_*,w}$ denote the space
$\Lambda_{\varphi_*,w}$ equipped with the Orlicz norm. Recall that the Orlicz norm in $L_{\varphi_*}(w)$ can be equivalently expressed by the Amemiya formula (e.g. \cite[p. 92, Theorem 10.2]{KrRut} or \cite[p. 7]{M})  which gives here
\[
 \|f\|^0_{\Lambda_{\varphi_*,w}}= \|f^*\|^0_{L_{\varphi_*} (w)} = \inf_{k>0} \frac{1}{k} \left(1 + \int_I \varphi_* (kf^*) w\right).
 \]
\end{definition}
For more information on Orlicz spaces we refer to \cite{BS, KPS, KrRut, LT2} and for Orlicz-Lorentz spaces to \cite{KR} and references there.
\begin{theorem}\label{th:dual}
Let $\varphi$ be N-function, and $\varphi_*$ be its complementary function. Then $M_{\varphi,w}' = \Lambda_{\varphi_*,w}^0$ that is the K\"othe dual of the class $M_{\varphi,w}$ is the classical Orlicz-Lorentz space $\Lambda_{\varphi_*,w}$, and the Orlicz norm on  $\Lambda_{\varphi_*,w}$ is dual to the homogeneous functional $\|\cdot\|_M$ in the sense that $ \|f\|^0_{\Lambda_{\varphi_*,w}} = \sup\left\{ \int_I fg: g\in M_{\varphi,w}, \|g\|_M\le 1\right\}$, for all $f\in  \Lambda_{\varphi_*,w}$.

\end{theorem}

\begin{remark}
When $W(t)=\int_0^tw(s)\,ds=\infty$ for $t>0$ the theorem states simply that $M'_{\varphi,w}=\{0\}$.
\end{remark}
\begin{proof}
i) We prove first that $\Lambda_{\varphi_*,w}\subset M'_{\varphi, w}$, and more precisely
\begin{equation}
f\in \Lambda_{\varphi_*,w}, g\in M_{\varphi, w} \implies fg\in L_1\hbox{ and } \|fg\|_1\le \|f\|^0_{\Lambda_{\varphi_*,w}} \|g\|_M.
\end{equation}
We have  $f^*\in L_{\varphi_*}(w)$  while $g^*/w\in L_{\varphi}(w)$. Using the Hardy-Littlewood inequality and the K\"othe duality of $(L_{\varphi_*}(w))'=L_{\varphi}(w)$ with respect to the new measure $wdt$, we have
\[ \int_I |fg| \le \int_I f^*g^*= \int_I f^*(g^*/w)w\le \|f^*\|^0_{L_{\varphi_*}(w)}\|g^*/w\|_{L_\varphi(w)}= \|f\|^0_{\Lambda_{\varphi_*,w}} \|g\|_M.  \]

ii) We prove now that the ``unit ball'' of $M_{\varphi,w}$ is 1-norming for $\Lambda_{\varphi_*,w}$ and its Orlicz norm, that is for all $f\in \Lambda_{\varphi_*,w}$,
\begin{equation}\label{1-norming}
 \|f\|^0_{\Lambda_{\varphi_*,w}} = \sup\left\{ \int_I fg: g\in M_{\varphi,w}: \|g\|_M\le 1\right\}.
\end{equation}
\par We prove the statement (\ref{1-norming}) first when $f$ is non-negative and decreasing. Since then $\|f\|^0_{\Lambda_{\varphi_*,w}}=\|f\|^0_{L_{\varphi_*}(w)}$, we can find for every $\eps>0$,  a non-negative $h\in L_\varphi(w)$ with the Luxemburg norm $\|h\|_{L_\varphi(w)}\le 1$ such that
\[
\int_I fhw\ge (1-\eps)\|f\|^0_{\Lambda_{\varphi_*,w}}.
\]
Let us show that $h$ may be chosen to be decreasing.  Note first that this does not result directly from the Hardy-Littlewood inequality because $f^*$ is perhaps not equimeasurable with $f$ for the measure $wdt$. We may assume that $W(t)<\infty$ for all $t>0$, since otherwise  $\Lambda_{\varphi_*,w}=\{0\}$ and the statement is trivial. Then with $J=W(I)=(0,b)$ for some $b\in (0,\infty]$,
\[\int_I fhw\,dt=\int_J (f\circ W^{-1})(h\circ W^{-1}) du\le \int_J (f\circ W^{-1})(h\circ W^{-1})^* du = \int_I f h_1 w du,\]
where $h_1= (h\circ W^{-1})^*\circ W$ is decreasing and equimeasurable with $h$ for the measure $wdt$. In particular
$\|h_1\|_{L_\varphi(w)}=\|h\|_{L_\varphi(w)}\le 1$. We put now $g=h_1w$. Then $g$  is decreasing and we have
\[\int_I fg\ge (1-\eps)\|f\|^0_{\Lambda_{\varphi_*,w}},\]
\[ \int_I \varphi\left(\frac gw\right)w=\int_I\varphi(h)w\,dt\le 1,\]
which implies that $\|g\|_M\le 1$. It shows (\ref{1-norming}) in the case when $f$ is non-negative and decreasing.

\par We deduce now the statement (\ref{1-norming}) for general $f$. Since $\Lambda_{\varphi_*,w}$ with Orlicz norm has the Fatou property it is sufficient to prove the statement when $f$ is a simple function with support of finite measure. In this case there is an invertible measure preserving automorphism of $I$ such that $|f|=f^*\circ \sigma$. Then for any $g\in M_{\varphi,w}$ we have
\[
\int_I f(\sign f) (g\circ \sigma) = \int_I |f|(g\circ \sigma) = \int_I (f^*\circ\sigma)( g\circ \sigma)=\int_I f^*g \hbox{ and } \|( \sign f )g\circ \sigma\|_M=\|g\|_M.
 \]
Therefore the statement results from the decreasing case.

iii) We show finally that $M'_{\varphi, w}\subset \Lambda_{\varphi_*,w}$.

Consider first the case where $W(t)<\infty$, for all $t\in I$. In this case $\Lambda_{\varphi_*,w}$ is not trivial and contains the bounded functions with support of finite measure.
\noindent Let $f\in M'_{\varphi, w}$. Consider any bounded function $h$ with support of finite measure such that $|h|\le |f|$. We have $h\in \Lambda_{\varphi_*,w}$ and by (\ref{1-norming}) and Lemma \ref{M'-norm} we get
\begin{align*}
\|h\|^0_{\Lambda_{\varphi_*,w}}&=\sup\left\{ \int_I |hg|: g\in M_{\varphi,w}: \|g\|_M\le 1\right\}\\
&\le \sup\left\{ \int_I |fg|: g\in M_{\varphi,w}: \|g\|_M\le 1\right\} := C(f)<\infty.
\end{align*}
Since $|f|=\sup \{h\in L^0: 0\le h\le |f|, h  \hbox{ bounded with finite measure support}\}$,  and $\Lambda_{\varphi_*,w}$ has the Fatou property it results that $f\in \Lambda_{\varphi_*,w} $.

Consider now the case where $W(t)=\infty$, $t\in I$. In this case $\Lambda_{\varphi_*,w}=\{0\}$ and we have to prove that $M'_{\varphi,w}$ is trivial too. It is sufficient to prove that this space contains no indicator function, and by Lemma
\ref{rearrangement-in-M'} that it does not contain $\chi_{[0,b]}$, $0<b\in I$. In order to do this we test this function on the functions
$f_u= (w\wedge w(u))\chi_{[0,b]}$, $u \in (0,b]$. For $c\ge 0$ we have
\begin{align*}
\int_0^b \varphi(cf_u/w)w& =\int_0^u \varphi(cw(u)/w(t))w(t)\,dt + \varphi(c)\int_u^b w(t)\,dt \\
&\le \varphi(c)uw(u)+\varphi(c)\int_u^b w(t)\,dt = \varphi(c)\int_0^b w(t)\wedge w(u)\,dt.
\end{align*}
Thus choosing  $c_u=\varphi^{-1}\left({1\over \int_0^b w(t)\wedge w(u)\,dt}\right)$ we have $\|c_uf_u\|_M\le 1$.  Note that $c_u\to 0$ when $u\to 0$. Then
\[\int_I \chi_{[0,b]} c_uf_u= c_u\int_0^b w(t)\wedge w(u)\,dt = {c_u\over \varphi(c_u)}\to \infty \hbox{ when } u\to 0,\]
since $\varphi$ is a $N$-function. This implies $\chi_{[0,b]}\not\in M'_{\varphi,w}$ by Lemma \ref{M'-norm} and completes the proof.
\end{proof}

\begin{example}
Here is an example of a decreasing function $w$ such that $W(t) < \infty$, $t\in I$ but $1/w$ does not verify the condition $\Delta_2$.  Consequently if $\alpha_\varphi>1$ the class $M_{\varphi,w}$ is not linear but its K\"othe dual is non trivial. Let $I=(0,1]$ and $w$ be defined by $w(t)=2^{k^2}$ when $t\in (4^{-(k+1)^2},4^{-k^2}]$, $k=0,1,\dots$. Then $w(t)\le t^{-1/2}$ for all $t\in I$ and so $W$ is finite on $I$ which implies that $M_{\varphi,w}'= \Lambda_{\varphi_*,w} \ne \{0\}$. Moreover $w(2t_k)= 2^{1-2k}w(t_k)$ for $t_k= 4^{-k^2}$, $k=1,2,\dots$, which implies that $1/w$ does not verify $\Delta_2$ condition.
\end{example}

\section{A new class of Orlicz-Lorentz Spaces}\label{sec:Spaces P}

It is well know that $\|\cdot\|_M$ (resp.,  $\|\cdot\|_m$) is a quasi-norm if the weight $w$ is regular. Here we will show more.  If $w$ is regular then we will define explicitly a norm on $M_{\varphi,w}$ equivalent to $\|\cdot\|_M$ in function case or a norm on $m_{\varphi,w}$ equivalent to  $\|\cdot\|_m$ in sequence case. In fact we will define a new class of r.i. Banach spaces induced by an Orlicz function $\varphi$ and a weight $w$. It will turn out that these spaces are in fact the K\"othe biduals of the classes $M_{\varphi,w}$.

\subsection{Definition and properties}

The formulas in Corollary \ref{cor:2} suggest the following definition.

\begin{definition} Let $\varphi$ be an Orlicz function and $w$ be a positive weight sequence or a weight function.  Define the following functionals
\begin{align*}
P(f) &= \inf\left\{\int_I \varphi\left(\frac{|f|}{v}\right) v:  v\prec w, v\ge 0\right\},\ \ \ \ f\in L^0,\\
p(x) &= \inf\left\{\sum_{n=1}^\infty \varphi\left(\frac{|x(n)|}{v(n)}\right) v(n):  v\prec w, v\ge 0\right\},\ \ \ \ x = \{x(n)\}.
\end{align*}
They correspond to function space $\mathcal{M}_{\varphi,w}$ and sequence space $\mathfrak{m}_{\varphi,w}$, defined respectively as the set of $f\in L^0$ and $x=\{x(n)\}$ such that
\[
\|f\|_\mathcal{M} = \inf\left\{\epsilon > 0: P
\left(\frac{f}{\epsilon}\right) \le 1\right\} < \infty \ \ \ \text{and}\ \ \
\|x\|_\mathfrak{m} = \inf\left\{\epsilon > 0: p\left(\frac{x}{\epsilon}\right) \le 1\right\} < \infty.
\]
\end{definition}

\begin{remark}\label{embeddings} By Corollary \ref{cor:2} it is clear that $P(f) \le M(f)$ and thus $M_{\varphi,w} \subset \mathcal{M}_{\varphi,w}$.  Now, let $h\ge 0$ be a function such that $h\prec w$. Then $t h^*(t) \le \int_0^t h^* = W(t)$, and so $h^*(t)\le w_1(t):=  W(t)/t$. Note that $w_1$ is also a decreasing weight in $I$ such that $w_1 \ge w$. Hence by Corollary \ref{cor:2},
\[
\int_I \varphi\left(\frac{|f|}{h}\right) h
\ge  M_1( f):= \int_I \varphi\left(\frac{f^*}{w_1}\right)w_1 = \inf\left\{\int_I \varphi\left(\frac{|f|}{v}\right)v: v^* \le w_1, v\ge 0\right\}.
\]
Thus  passing to infimum when $h\prec w$ we get $P(f) \ge M_1(f)$, and consequently we have
\begin{equation}\label{inclusions}
M_{\varphi,w}\subset \mathcal{M}_{\varphi,w}\subset M_{\varphi,w_1}.
\end{equation}
Analogously $m_{\varphi,w} \subset \mathfrak{m}_{\varphi,w}\subset m_{\varphi,w_1}$. Moreover
\begin{equation*}
\|f\|_M\ge \|f\|_\mathcal{M}\ge \|f\|_{M_1},
\end{equation*}
and in particular the functional $\|f\|_\mathcal{M}$ is faithful that is $\|f\|_\mathcal{M}=0$ implies $f=0$ as an element of $L^0$.

\begin{proposition}\label{prop:equiv}
If the weight $w$ is regular then the classes $\mathcal{M}_{\varphi,w}$ and $M_{\varphi,w}$ coincide and the associated functionals $\|\cdot\|_\mathcal{M}$ and $\|\cdot\|_M$ are equivalent.
\end{proposition}
\begin{proof}
If the weight $w$ is regular, the weights $w$ and $w_1$ are equivalent.  Hence $M_{\varphi,w}=M_{\varphi, w_1}$ and thus by (\ref{inclusions}), $\mathcal{M}_{\varphi,w}=M_{\varphi,w}$ (analogously $m_{\varphi,w} = \mathfrak{m}_{\varphi,w}$ when the weight $w$ is regular). More precisely if $w_1\le C w$, $C\ge 1$,  then $M_1(f)\ge CM(C^{-1}f)\ge M(C^{-1}f)$, and thus $\|f\|_{M_1}\ge C^{-1}\|f\|_M$.
\end{proof}
\noindent  A partial converse to Proposition \ref{prop:equiv} will be given later in Proposition \ref{prop:equality P=M}.
\end{remark}

We show now that $P(f)$ and $p(x)$ are rearrangement invariant.

\begin{proposition}\label{prop:ri}
 For any function $f\in L^0$ and a sequence $x = \{x(n)\}$ we have $P(f) = P(f^*)$ and $p(x) = p(x^*)$.

\end{proposition}
\begin{proof} We provide the proof only in the case when $I=(0,\infty)$. Let $f\in L^0$.

First we show that $P(f^*) \ge P(f)$. Let $v\prec w$ and assume $\int_I \varphi\left(\frac{f^*}{v}\right) v < \infty$. It follows that $\lim_{t\to\infty}f^*(t) = 0$. Indeed, if we assume that $\inf_{t\ge 0} f^*(t) = K>0$ then setting $B= \{t: v(t) \le b\}$, $b> 0$, we have
$|B| \varphi\left(\frac{K}{b}\right) b \le \int_B \varphi\left(\frac{K}{v}\right) v
\le \int_I \varphi\left(\frac{f^*}{v}\right) v <\infty.$ Hence $|B| < \infty$ for any $b\ge 0$, and so $|\{t: v(t) > b\}| = \infty$ for every $b\ge 0$. Consequently $v^* =\infty$ on $(0,\infty)$, which contradicts the assumption $v\prec w$.

Thus by Ryff's theorem there exists an onto and measure preserving transformation $\tau: \supp f \to \supp f^*$ such that $|f| = f^* \circ \tau$ on $\supp f$. If $|\supp f| < \infty$ then $|\supp f| = |\supp f^*|$ and we extend $\tau$ to a measure preserving transformation from $I$ onto $I$. If $|\supp f|= \infty$ then $\supp f^* = I$. In the first case define $v_1 =v\circ \tau$, in the second case let $v_1(t) = v \circ \tau(t)$ for $t\in\supp f$ and $v_1(t) = 0$ for $t\in (\supp f)^c$. Since the range of $\tau$ in both cases is $(0,\infty)$, so $v_1\sim v$. Moreover in the first case,
\[
\int_I \varphi\left(\frac{f^*}{v}\right)v = \int_I \varphi\left(\frac{f^*\circ \tau}{v\circ\tau}\right) v\circ\tau =\int_I \varphi\left(\frac{|f|}{v_1}\right)v_1,
\]
and in the second case applying our convention we also get
\[
\int_I \varphi\left(\frac{f^*}{v}\right)v = \int_{\supp f} \varphi\left(\frac{|f|}{v\circ\tau}\right) v\circ\tau =\int_I \varphi\left(\frac{|f|}{v_1}\right)v_1.
\]
Summarizing, for every $v\prec w$ we found $v_1\prec w$ such that
$\int_I \varphi\left(\frac{f^*}{v}\right)v = \int_I \varphi\left(\frac{|f|}{v_1}\right)v_1$, which shows that $P(f^*) \ge P(f)$.

Now we will show that $P(f) \ge P(f^*)$.
Let's consider first $f\in L^0$ such that its range is countable.
Let $v\prec w$ and $\int_I \varphi\left(\frac{|f|}{v}\right)v < \infty$. Let $A = \{t: |f(t)| > K\}$ where $K=\lim_{t\to\infty}f^*(t)$.

Suppose $|A| = \infty$. Then since $f$ has countable range there exists an onto and measure preserving transformation $\tau: (0,\infty) \to A$ such that $|f|\circ \tau = f^*$ on $(0,\infty)$. Therefore
\begin{equation}\label{eq:31}
\int_I \varphi\left(\frac{|f|}{v}\right)v \ge \int_A \varphi\left(\frac{|f|}{v}\right)v =
\int_I \varphi\left(\frac{|f|\circ\tau}{v\circ\tau}\right)v\circ\tau =
\int_I \varphi\left(\frac{f^*}{v\circ\tau}\right)v\circ\tau.
\end{equation}
Since the range of $\tau$ is equal to $A$, setting $v_1=v \circ\tau$ we have $v_1^* \le v^*$, and hence $v_1\prec w$.

Now let $|A|<\infty$. Setting $B=\{t: |f(t)| = K\}$, we have $|B|=\infty$. We find
onto and measure preserving transformations $\tau_1: (0,|A|) \to A$ and
$\tau_2: (|A|, \infty) \to B$ such that $|f|\circ \tau_1(t) = f^*(t)$ for $t\in (0,|A|)$, and $|f|\circ \tau_2(t) = f^*(t)=K$ for $t\in (|A|, \infty)$. Thus
$\tau = \tau_1|_{(0,|A|)} + \tau_2|_{(|A|, \infty)}$ is a measure preserving mapping from $(0,\infty)$ onto $A \cup B$ and such that $|f|\circ \tau = f^*$ on $I$. Then
\begin{equation}\label{eq:32}
\int_I \varphi\left(\frac{|f|}{v}\right)v \ge \int_{A\cup B} \varphi\left(\frac{|f|}{v}\right)v =
\int_I \varphi\left(\frac{|f|\circ\tau}{v\circ\tau}\right)v\circ\tau =
\int_I \varphi\left(\frac{f^*}{v\circ\tau}\right)v\circ\tau.
\end{equation}
We have that $(v\circ \tau)^* \le v^*$ and so $v_1 = v\circ \tau \prec w$.

By (\ref{eq:31}) and (\ref{eq:32}), 
for any $v\prec w$ we can find $v_1 \prec w$ such that
\begin{equation}\label{eq:33}
\int_I \varphi\left(\frac{|f|}{v}\right)v \ge
\int_I \varphi\left(\frac{f^*}{v_1}\right)v_1\ge P(f^*).
\end{equation}
Let now $f\in L^0$ and $v\prec w$ be such that $\int_I\varphi\left(\frac{|f|}{v}\right)v< \infty$. There exists a sequence $\{g_n\}$ of countable valued and measurable functions such that
\[
|f| \le g_n \le v\varphi^{-1} \left(\left(1 + \frac{1}{n}\right) \varphi\left(\frac{|f|}{v}\right)\right) \ \ \ \text{and} \ \ \ g_n \downarrow |f|\ \text{a.e.}.
\]
Hence $\varphi\left(\frac{g_n}{v}\right)v\downarrow  \varphi\left(\frac{|f|}{v}\right)v$ and $\varphi\left(\frac{g_n}{v}\right)v\le 2 \varphi\left(\frac{|f|}{v}\right)v$ a.e., and by the Lebesgue convergence theorem,
\[
\lim_{n\to\infty}\int_I\varphi\left(\frac{g_n}{v}\right)v =  \int_I\varphi\left(\frac{|f|}{v}\right)v.
\]
Hence for any $\epsilon > 0$ there exists $m\in\mathbb{N}$ such that
\[
\int_I\varphi\left(\frac{|f|}{v}\right)v \ge  \int_I\varphi\left(\frac{g_m}{v}\right)v - \epsilon.
\]
Now by (\ref{eq:33}) there is $v_1 \prec w$ and such $\int_I\varphi\left(\frac{g_m}{v}\right)v \ge  \int_I\varphi\left(\frac{g_m^*}{v_1}\right)v_1$. But $g_m^* \ge f^*$, and so
\[
\int_I\varphi\left(\frac{|f|}{v}\right)v \ge  \int_I\varphi\left(\frac{g^*_m}{v_1}\right)v_1 - \epsilon
\ge  \int_I\varphi\left(\frac{f^*}{v_1}\right)v_1 - \epsilon \ge P(f^*) - \epsilon.
\]

\end{proof}

The next corollary provides alternative formulas for $\|\cdot\|_P$ and $\|\cdot\|_p$.
It follows by Proposition \ref{prop:ri} and Theorem \ref{integral:ineq} in function case and Theorem \ref{th:infinite:equality1} in sequence case.

\begin{corollary}\label{cor:3}
For any $f\in L^0$ and $x= \{x(n)\}$,
\[
P(f) = \inf\left\{\int_I \varphi\left(\frac{f^*}{v}\right) v: v\prec w, v\downarrow \right\}, \ \ \
p(x) = \inf\left\{ \sum_{n=1}^\infty \varphi\left( \frac{x^*(n)}{v(n)}\right) v(n): v\prec w, v\downarrow\right\},
\]
where $v\downarrow$ denotes a non-negative decreasing function or sequence.

\end{corollary}

\begin{lemma}\label{lem:Fatou}
For any non-negative functions $f_n, f \in L^0$ such that $f_n \uparrow f$ and $\sup_n P(f_n) < \infty$, it holds $P(f_n) \uparrow P(f)$. The similar statement holds in sequence case.

\end{lemma}
\begin{proof}
By Corollary \ref{cor:3}, we can assume that each $f_n$ is decreasing and for every $n\in \mathbb{N}$ there exists a non-negative decreasing $v_n$ such that $v_n\prec w$ and $ \frac{1}{n} + P(f_n) \ge \int_I \varphi\left(\frac{f_n}{v_n}\right) v_n$, $n\in \mathbb{N}$. By submajorization of $v_n$ by $w$ we have $v_n(t) \le W(t)/t$, $n\in\mathbb{N}$. Hence the family $\{v_n\}$ is a sequence of decreasing functions bounded uniformly by a decreasing function on $I$.  So, by Helly's selection theorem \cite[Chapter 8, section 4]{N}, for some subsequence we have that $\lim_n v_n (t) = v(t)$ exists a.e. Obviously $v$ is decreasing on $I$ and $\int_0^t v \le \underline{\lim}_n \int_0^t v_n \le \int_0^t w$, and so $v\prec w$. Then by Fatou's Lemma we get
\[
P(f) \ge \lim_n P(f_n) =\lim_n \int_I \varphi\left(\frac{f_n}{v_n}\right) v_n \ge \int_I \lim_n \varphi\left(\frac{f_n}{v_n}\right) v_n = \int_I \varphi\left(\frac{f}{v}\right) v \ge P(f),
\]
and the proof is completed.

\end{proof}

Before the next proof, observe that the function $\psi(s,t)= \varphi(s/t)t,\, s,t > 0$, is convex with respect to two variables. In fact letting $s_i, t_i >0$, $i=1,2$, by convexity of $\varphi$, $\varphi\left(\frac{s_1 + s_2}{t_1 + t_2}\right) \le
\frac{t_1}{t_1 + t_2} \varphi\left( \frac{s_1}{t_1}\right)  + \frac{t_2}{t_1 + t_2}\varphi\left(\frac{s_2}{t_2}\right)$.  Hence
$\psi\left(\frac{s_1+ s_2}{2}, \frac{t_1 + t_2}{2}\right) \le \frac12 \varphi\left(\frac{s_1}{t_1}\right) t_1 + \frac12 \varphi\left(\frac{s_2}{t_2}\right) t_2 = \frac12 \psi(s_1,t_1) + \frac12 \psi(s_2,t_2)$.

\begin{theorem}\label{norm}
The spaces $(\mathcal{M}_{\varphi,w}, \|\cdot\|_\mathcal{M})$ and $(\mathfrak{m}_{\varphi,w}, \|\cdot\|_\mathfrak{m})$ are rearrangement invariant Banach spaces satisfying the Fatou property.

\end{theorem}
\begin{proof} We prove it only in function case. We show first that $\|\cdot\|_\mathcal{M}$ is a norm. The fact that this functional is faithful was observed in Remark \ref{embeddings}. It remains only to show that the functional $P$ is convex for proving that its the homogeneous functional $\|\cdot\|_\mathcal{M}$ satisfies the triangle inequality. Let $P(f_i) < \infty$ for $i=1,2$, and $\epsilon > 0$ be arbitrary. There exist $h_i$, $i=1,2$, such that $h_i\prec w$ and
\[
 P(f_i) + \epsilon \ge \int_I \varphi\left(\frac{|f_i|}{h_i}\right) h_i, \ \ \ \ i=1,2.
\]
For $h=\frac{h_1 + h_2}{2}$, by subadditivity $(h_1 + h_2)^* \prec h_1^* + h_2^*$ we get $h\prec w$. Thus, in view of convexity of the function $(t,s)\mapsto \varphi(t/s)s,\, s,t > 0$, it follows
\begin{align*}
 P\left(\frac{f_1 + f_2}{2}\right) &\le \int_I \varphi\left(\frac{|f_1+f_2|}{2h}\right) h\\ & \le
\frac12 \int_I \varphi\left(\frac{|f_1|}{h_1}\right) h_1 + \frac12 \int_I \varphi\left(\frac{|f_2|}{h_2}\right) h_2\\
&\le \frac12\left(P(f_1) + P(f_2)\right) + \epsilon.
\end{align*}
Hence $P$ is convex and so $\|\cdot\|_\mathcal{M}$ is a norm on $\mathcal{M}_{\varphi,w}$. It is rearrangement invariant in view of Proposition \ref{prop:ri}.

Finally let $f_n\in \mathcal{M}_{\varphi,w}$, $f\in L^0$ be non-negative,  $f_n\uparrow f$ and $\sup_n\|f_n\|_\mathcal{M} = K < \infty$. Then
$P\left(\frac{f_n}{K}\right) \le 1$ for all $n\in \mathbb{N}$, and
 by Lemma \ref{lem:Fatou}, $P\left(\frac{f_n}{K}\right) \uparrow P\left(\frac{f}{K}\right) \le 1$. Thus
 $\|f\|_{\mathcal{M}} \le \sup_n \|f_n\|_\mathcal{M}$, and so $\|f\|_{\mathcal{M}} = \sup_n \|f_n\|_{\mathcal{M}}$. Therefore  $\mathcal{M}_{\varphi,w}$ has the Fatou property, and thus it is complete in view of \cite[Theorem  1.6]{BS}.
 \end{proof}

 \begin{remark}\label{rem:concavity of P}
If $\varphi$ is equivalent to a $p$-concave Orlicz function for some $1<p<\infty$ then $\mathcal{M}_{\varphi,w}$ is a $p$-concave Banach lattice. The proof is similar to that of Proposition \ref{prop:concavity} and Corollary \ref{cor:316}.
\end{remark}

\subsection{When do $\mathcal{M}_{\varphi,w}$ and $M_{\varphi,w}$ coincide?}

We have seen in Proposition \ref{prop:equiv} that if $w$ is regular then the classes $\mathcal{M}_{\varphi,w}$ and $M_{\varphi,w}$ coincide and the associated functionals  $\|\cdot\|_\mathcal{M}$ and $\|\cdot\|_M$ are equivalent. In this subsection we shall give a partial converse in the case when $\alpha_\varphi>1$. For this aim we shall compare the fundamental functions of $\mathcal{M}_{\varphi,w}$ and $M_{\varphi,w}$.

\begin{proposition}\label{fundamental-P-space}
Let $w$ be a weight function such that $W(t) < \infty$, $t\in I$. The fundamental function of the space $\mathcal{M}_{\varphi,w}$ is given by
\[
F_\mathcal{M}(t)= \|\chi_{(0,t)}\|_\mathcal{M} = {t\over W(t)\varphi^{-1}(1/W(t))}, \ \ \ t\in I.
\]
The analogous formula is also valid in the sequence space $\mathfrak{m}_{\varphi,w}$.
\end{proposition}
\begin{proof}
For every constant $c>0$ and every decreasing weight $v\prec w$ over  $(0,t)$ we have
by Jensen's inequality
\[
\int_0^t \varphi\left(\frac{1}{c v(s)}\right) \frac{v(s)}{V(t)} \, ds
\ge \varphi\left(\int_0^t \frac{1}{c V(t)}\, ds\right) =  \varphi\left(\frac{t}{c V(t)}\right),
\]
and since $V\le W$ and for $\alpha>0$ the function $t\mapsto t\varphi(\alpha/t)$  is decreasing for $t>0$, we get
\[
\int_0^t \varphi\left(\frac{1}{c v(s)}\right) v(s) \, ds\ge W(t)\varphi\left(\frac{t}{c W(t)}\right).
\]
Taking the infimum with respect to decreasing weights $v\prec w$, by Corollary \ref{cor:3} we obtain
\[
P\left(\frac 1c \chi_{(0,t)}\right)\ge W(t)\varphi\left(\frac{t}{c W(t)}\right).
\]
The weight $v_0=\frac{W(t)}t\chi_{(0,t)}$ is decreasing with support $(0,t)$ and since $W(u)/u$ is decreasing we have for $0\le u\le t$,
\[
V_0(u)=\int_0^u v_0(s)\,ds= \frac ut W(t)\le W(u),
\]
while for $u > t$,
\[
V_0(u)=V_0(t)=W(t)\le W(u).
\]
Hence $v_0\prec w$, and so
\[
P\left(\frac 1c \chi_{(0,t)}\right)\le \int_0^t \varphi\left(\frac 1{cv_0(s)}\right)v_0(s)\,ds= \varphi\left(\frac t{c \,W(t)}\right) W(t).
\]
Consequently for every $c>0$ we obtain the equality
\[
P\left(\frac 1c \chi_{(0,t)}\right)= W(t)\varphi\left(\frac{t}{c W(t)}\right),
\]
which implies that $F_\mathcal{M}(t)=\|\chi_{(0,t)}\|_\mathcal{M}$ is the unique solution $c$ of $P\left(\frac 1c \chi_{(0,t)}\right)=1$, and gives the desired formula.
\end{proof}

\begin{proposition}\label{prop:equality P=M}
 Let $\alpha_\varphi>1$ and $w$ be a weight function such that $W(t) < \infty$, $t\in I$. If the norm $\|\cdot\|_\mathcal{M}$ of $\mathcal{M}_{\varphi,w}$ and the functional $\|\cdot\|_M$ of $M_{\varphi,w}$ are equivalent on $\mathcal{M}_{\varphi,w}$ then the weight $w$ is regular. In particular the spaces $\mathcal{M}_{\varphi,w}$ and $M_{\varphi,w}$ are equal with equivalent quasi-norms if and only if $w$ is regular. A similar statement is valid for sequence spaces.
 \end{proposition}

 \begin{proof} We will conduct the proof only in function case when $I=(0,\infty)$.
 By hypothesis, for some constant $C$ and every $t\in I$, $F_M(t)\le CF_\mathcal{M}(t)$. Thus in view of inequalities (\ref{estim F_M}) and Proposition \ref{fundamental-P-space},
\[
 \frac 1{w(t)\varphi^{-1}\left(\frac 1{tw(t)}\right)}\le F_M(2t)\le CF_\mathcal{M}(2t)= \frac {2Ct}{W(2t)\varphi^{-1}\left(\frac 1{W(2t)}\right)}, \ \ \ t>0,
 \]
 that is  for $t>0$,
\begin{align}\label{control of W}
W(2t)\varphi^{-1}\left(\frac 1{W(2t)}\right)\le 2C tw(t) \varphi^{-1}\left(\frac 1{tw(t)}\right).
\end{align}
Since $\alpha_\varphi>1$ so $\beta_{\varphi^{-1}} = 1/\alpha_\varphi<1$, and hence for some $\varepsilon>0$ and $K>0$ we have for every $\lambda\ge 1$ and $u>0$,
\[
\varphi^{-1}(\lambda u)\le K\lambda^{1-\varepsilon}\varphi^{-1}(u).
\]
Thus in view of $\frac{W(2t)}{tw(t)}>1$ we have
\[
2C tw(t) \varphi^{-1}\left(\frac 1{tw(t)}\right)<2CK tw(t)\left( \frac{W(2t)}{tw(t)}\right)^{1-\varepsilon}\varphi^{-1}\left(\frac 1{W(2t)}\right).
\]
Combining the above  with inequality (\ref{control of W}) we get easily
\[
\frac {tw(t)}{W(2t)}\ge \frac 1{(2CK)^{1/\varepsilon}}, \ \ \ t > 0.
\]
Eventually, since $W(2t)\ge W(t)$ this implies that $w$ is regular and the proof is completed.
 \end{proof}

 \subsection{K\"othe duality}

\begin{proposition}\label{duality2}
Assume that $\varphi$ is $N$-function and $W(t)<\infty$ for $t\in I$. Then the K\"othe dual of the space $\mathcal{M}_{\varphi,w}$ is the Orlicz-Lorentz space $\Lambda_{\varphi_*, w}$ equipped with its Orlicz norm, that is $\mathcal{M}'_{\varphi,w} = \Lambda^0_{\varphi_*,w}$ with equality of norms.
\end{proposition}

\begin{proof}
Since $M_{\varphi,w}\subset \mathcal{M}_{\varphi,w}$, and the inclusion has norm not greater than one, we have
\[ \mathcal{M}'_{\varphi,w}\subset M'_{\varphi,w}=\Lambda^0_{\varphi_*,w},\]
where the inclusion has norm not greater than one, and the equality is  isometric by Theorem \ref{th:dual}.

Conversely let us prove that $\Lambda^0_{\varphi_*,w}\subset \mathcal{M}'_{\varphi,w}$, and the inclusion has norm not greater than one. First we note that for any $g\in \Lambda^0_{\varphi_*,w}$ and any non-negative weight $v$ with $v\prec w$ we have
\[ \int_I \varphi_*(g^*)v \le \int_I \varphi_*(g^*)v^*\le \int_I \varphi_*(g^*)w,\]
since the function $\varphi_*(g^*)$ is decreasing. Thus $\Lambda^0_{\varphi_*,w}\subset \Lambda^0_{\varphi_*,v}$, and from the Amemiya formula for the Orlicz norm by Definition \ref{def} we get
\[ \|g\|^0_{\Lambda_{\varphi_*, v}}= \|g^*\|^0_{L_{\varphi_*}(v)}\le \|g^*\|^0_{L_{\varphi_*}(w)} = \|g\|^0_{\Lambda_{\varphi_*, w}}.\]
Fix now $g\in \Lambda^0_{\varphi_*,w}$. Then for any $v\prec w$ and  $h\in M_{\varphi,v}$, we have
\[\int_I |hg|\le \|h\|_{M_{\varphi,v}}\|g\|^0_{\Lambda_{\varphi_*,v}}\le \|h\|_{M_{\varphi,v}}\|g\|^0_{\Lambda_{\varphi_*,w}}.\]
By Corollary \ref{cor:3}, if $h\in \mathcal{M}_{\varphi,w}$ has norm $\|h\|_{\mathcal{M}_{\varphi,w}}<1$ then there exists some decreasing $v$ such that $v\prec w$ and $\int_I \varphi(h^*/v)v<1$. Hence $\|h\|_{M_{\varphi,v}}\le 1$ and so
\[\int_I |hg|\le\|g\|_{\Lambda^0_{\varphi_*,w}}.\]
This shows that $g\in \mathcal{M}'_{\varphi,w}$ with norm $\|g\|_{\mathcal{M}'_{\varphi,w}}\le \|g\|_{\Lambda^0_{\varphi_*,w}}$.
\end{proof}

\begin{corollary}[K. Le\'snik]\label{cor:lesnik}
Assume that $\varphi$ is $N$-function and $W(t)<\infty$ for $t\in I$.
Then the space  $\mathcal{M}_{\varphi,w}$ is equal to the K\"othe dual of the Orlicz-Lorentz space $\Lambda^0_{\varphi_*, w}$, that is $(\Lambda^0_{\varphi_*,w})' = \mathcal{M}_{\varphi,w}$ with equality of norms.
\end{corollary}

\begin{proof}
Since $\mathcal{M}_{\varphi,w}$ is a K\"othe function space in the sense of \cite[1.b.17]{LT2} and has the Fatou property, it holds by  \cite[p. 30, Remark 2]{LT2} that it is equal to its K\"othe bidual with equal norm. Then by Proposition \ref{duality2}
\[  \mathcal{M}_{\varphi,w}=\mathcal{M}_{\varphi,w}''=(\Lambda^0_{\varphi_*,w})'.\]
\end{proof}

Before we state the next result recall the definition of the Banach envelope of a quasi-Banach space $X$. Denote by $X^*$ the dual space to $X$, that is the (Banach) space of bounded linear functionals. Let us define a functional on $X$ by
\[ \trnorm x=\sup\{ |f(x)|: f\in X^*,\hbox{ and } \|f\|\le 1\}. \]
If $X^*$ separates the points of $X$ then $\trnorm \cdot$ is a norm on $X$. Then the {\it Banach envelope} $\widehat X$ of $X$ is simply the completion of the normed linear space $(X, \trnorm \cdot)$   \cite[pp. 27-28]{KPR}.

\begin{corollary}\label{cor:envelope}
Let $\varphi$ be $N$-function and $W(t)<\infty$ for $t\in I$.
The space $\mathcal{M}_{\varphi,w}$ is the K\"othe bidual of $M_{\varphi,w}$. Consequently if $M_{\varphi,w}$ is a linear space and $\varphi$ verifies condition $\Delta_2$, then $\mathcal{M}_{\varphi,w}$ is the Banach envelope of $M_{\varphi,w}$.
\end{corollary}

\begin{proof}
The first assertion is clear by Theorem \ref{th:dual} and Corollary \ref{cor:lesnik}. For the second one  by Lemma \ref{lem:quasi-norm} the assumption that $M_{\varphi,w}$ is linear implies that $\|\cdot\|_M$ is a quasi-norm.
 Now if $\varphi$ satisfies condition $\Delta_2$, then $M_{\varphi,w}$ is order continuous by  the general results on the symmetrization of Banach function spaces (see \cite[p. 279]{KR}). It results that the dual space  $M_{\varphi,w}^*$  coincides with the K\"othe dual $M'_{\varphi,w}$, and $\trnorm \cdot$ is simply the norm induced on $M_{\varphi, w}$ by the norm of its K\"othe bidual $\mathcal{M}_{\varphi,w}$. On the other hand, the condition  $\Delta_2$ for $\varphi$ implies that $\beta_\varphi < \infty$ and so there exists $\beta_\varphi < p < \infty$ such that the function $\varphi(t^{1/p})/t$ is pseudo-decreasing, that is $\varphi(t^{1/p})/t \ge C\varphi(u^{1/p})/u$ for $0<t<u$ and some $C>0$. Hence $\varphi(t^{1/p})$ is equivalent to a concave function \cite{MO, KMP1}, and by Remark \ref{rem:concavity of P},    the Banach function space $\mathcal{M}_{\varphi,w}$ is $p$-concave.  Then $\mathcal{M}_{\varphi,w}$ cannot contain an order isomorphic copy of $\ell_\infty$ and thus is order continuous (see e.g.  \cite{KA}).  Since $M_{\varphi, w}$ contains all simple integrable functions and $\mathcal{M}_{\varphi,w}$ is order continuous, then $\mathcal{M}_{\varphi,w}$ is the closure of all simple integrable functions \cite{BS}, and thus  we have $\widehat{M_{\varphi,w}}=\mathcal{M}_{\varphi,w}$.
\end{proof}

\begin{proposition}
Assume that $\varphi$ is  $N$-function and that $W(t)<\infty$, $t\in I$. The following assertions are equivalent.
\begin{itemize}
\item[(i)] $M_{\varphi,w}$ is a linear space and has a norm equivalent to the functional $\|\cdot \|_M$.
\item[(ii)] $M_{\varphi,w}=\mathcal{M}_{\varphi,w}$ and the functional $\|\cdot \|_M$ is equivalent to the norm $\|\cdot\|_\mathcal{M}$.
\end{itemize}
\noindent If moreover $\alpha_\varphi>1$ these conditions are equivalent to
\begin{itemize}
\item[(iii)] The weight $w$ is regular.
\end{itemize}
\end{proposition}

\begin{proof}
(ii) $\implies$ (i) is clear since $\mathcal{M}_{\varphi,w}$ is a Banach function space.

(i) $\implies$ (ii) Notice that the quasi-norm $\|\cdot\|_M$ has the Fatou property, by the general results on the symmetrization of quasi-Banach function spaces \cite[p. 279]{KR}. Thus if $\trnorm \cdot$ is an equivalent norm to $\|\cdot\|_M$ on $M_{\varphi,w}$, then $ \trnorm f _1 = \inf\{\trnorm g : |f|\le g\}$ is a norm equivalent to $\trnorm \cdot$ and so to $\|\cdot\|_1$, which preserves the order structure  of $M_{\varphi,w}$ and satisfies the isomorphic Fatou property, that is if $f_n\in M_{\varphi,w}$, $f_n \uparrow f$ a.e., and $\sup_n \trnorm {f_n} _1 < \infty$ then $f\in M_{\varphi,w}$ and $\sup_n \trnorm {f_n}_1 \le C \trnorm f _1$ for some $C>0$ depending only on the norm $\trnorm\cdot_1$. It is well known then that $M_{\varphi,w}$ can be renormed with an order compatible norm satisfying the usual (isometric) Fatou property, namely
\[ \trnorm{f}_2=\inf\{\lim_n\trnorm{f_n}_1: 0\le f_n\uparrow |f|\}\]
 (see e.g. \cite[pp. 446-452]{Z} where the isomorphic Fatou property is called the weak Fatou property). Then $(M_{\varphi,w}, \trnorm \cdot_2)$ becomes a Banach function space with the Fatou property. Since $M_{\varphi,w}$ as well as its K\"othe dual $\Lambda_{\varphi_*,w}$ contain the indicator functions of integrable sets, they become K\"othe function spaces in the sense of \cite[1.b.17]{LT2}. Hence $(M_{\varphi,w}, \trnorm \cdot_2)''= (M_{\varphi,w}, \trnorm \cdot_2)$ isometrically. By Corollary \ref{cor:envelope} we also have that $(M_{\varphi,w}, \| \cdot\|_M)'' = (\mathcal{M}_{\varphi,w}, \|\cdot\|_{\mathcal{M}})$. The equivalence of $\trnorm \cdot_M$ and $\|\cdot\|_M$ propagates to their dual and bidual norms, so finally $M_{\varphi,w} = \mathcal{M}_{\varphi,w}$ as sets and the quasinorm $\|\cdot\|_M$ is equivalent to the norm $\|\cdot\|_{\mathcal{M}}$.

Finally the equivalence (ii) $\iff$ (iii) when $\alpha_\varphi>1$ is simply Proposition \ref{prop:equality P=M}.
\end{proof}

\begin{example}
Here is an example of a decreasing weight function $w$ with $W(t) < \infty$, $t\in I$, and such that $1/w$ verifies condition $\Delta_2$ but $w$ is not regular.  Consequently if $\alpha_\varphi>1$ then the class $M_{\varphi,w}$ is linear and $\|\cdot\|_M$ is a quasi-norm, but it does not admit an equivalent norm.

Let $I=(0,1)$ and $w$ be defined by $w(t)=\max(2^{-(k+1)^2}t^{-1},2^{k^2})$ when $t\in (4^{-(k+1)^2},4^{-k^2}]$, $k=0,1,\dots$. Then $w(t)\le t^{-1/2}$ for all $t\in I$, and so $W(t)<\infty$ for $t\in I$. We also have that $w(st)\ge s^{-1}w(t)$ for all $s>1$, $t>0$ such that $st\le 1$, while $w(st_k)= s^{-1}w(t_k)$ for $t_k= 4^{-k^2}$, $k=1,2\dots$, and $1\le s\le w(t_{k})/w(t_{k-1})= 2^{2k-1}$. This implies that $\alpha_w =-1$. Thus $w$ is not regular by \cite[Lemma 6]{KMP}, while $1/w$ verifies $\Delta_2$ condition since $\beta_{1/w}= - \alpha_w = 1 < \infty$.
\end{example}

  \begin{remark}\label{rem:3}
Let's recall the space which was discussed in \cite{KM}. For an Orlicz function $\varphi$ and a weight function $w$ let  $S_{\varphi,w}$ be the space of $f\in L^0$ such that
\[
\|f\|_S = \inf\{\epsilon>0: S(f/\epsilon) \le 1\} < \infty, \ \ \ \text{where}\ \ \
S(f) = \int_I \varphi \left(\frac{\int_0^t f^*}{W(t)}\right) w(t)\,dt.
\]
Then $(S_{\varphi,w}, \|\cdot\|_S)$ is a r.i. Banach space (it is called $M_{\varphi,w}$ in \cite{KM}). It follows from \cite[Theorem 3.1]{KM} that $S_{\varphi,w}=\Lambda'_{\varphi_*w}$  with equivalent norms whenever  $\varphi$ and its complementary function
$\varphi_*$, satisfy condition $\Delta_2$ and $W(\infty)=\infty$ (in the case where the interval $I$ is infinite). However they are not equal neither to $M_{\varphi,w}$ nor $\mathcal{M}_{\varphi,w}$ without these  assumptions.

 Let's take for instance $w\equiv 1$ on $I$, and $\varphi(t) = t$, $t\ge 0$. Clearly $\varphi_*$ does not satisfy $\Delta_2$. Moreover $S_{\varphi,w} = L Log L$ if $a=1$ and $S_{\varphi,w} = \{0 \}$ in case when $a=\infty$ \cite{BS}.
However $M_{\varphi,w} = \mathcal{M}_{\varphi,w} = L_1$ with equality of norms.
\end{remark}


\begin{thebibliography}{99}

\bibitem{BS}
C. Bennet and R. Sharpley, \emph{Interpolation of operators}, Academic Press, 1988.


\bibitem{CKMP}
M. Cwikel, A. Kami\'nska, L. Maligranda and L. Pick, \emph{Are generalized Lorentz "spaces" really spaces?}, Proc. Amer. math. Soc \textbf{132}~(2004), No. 12,  3615--3625.


\bibitem{HKM}
H. Hudzik, A. Kami\'nska and M. Masty{\l}o, \emph{On the dual of
Orlicz-Lorentz space},
Proc. Amer. Math. Soc. \textbf{130}~(2002), no. 6, 1645--1654.

\bibitem{KPR}
N.J. Kalton, N.T. Peck and J.W. Roberts, \emph{An $F$-space sampler}, London Mathematical Society
Lecture Note Series, 89. Cambridge University Press, Cambridge, 1984.


\bibitem{KLR}
A. Kami\'nska, K. Le\'snik and Y. Raynaud, \emph{Dual spaces to Orlicz-Lorentz spaces}, preprint.


 \bibitem{KMP}
 A. Kami\'nska, L. Maligranda and L.E. Persson, \emph{Convexity, concavity, type and cotype of Lorentz spaces},  Indag. Mathem., N.S. \textbf{9}~(1998), 367--382.

\bibitem{KMP1}
A. Kami\'nska, L. Maligranda and L.E. Persson, \emph{ Indices and regularizations of measurable functions},  Proceedings
      of the Conference on Banach Function Spaces held in Poznan in 1998,  Lecture Notes in Pure and
      Applied Mathematics, Marcel Dekker, Inc; vol. 213, (2000), 231-246.


\bibitem{KM}
A. Kami\'nska and M. Masty{\l}o, \emph{Abstract duality Sawyer formula and its applications}, Monatsh. Math. \textbf{151}~(2007), no. 3, 223--245.

\bibitem{KR}
A. Kami\'nska and Y. Raynaud, \emph{Isomorphic copies in the lattice E and its symmetrization $E^{(*)}$ with applications to Orlicz-Lorentz spaces}, J. Funct. Anal. \textbf{257} (2009), no. 1, 271--331.

\bibitem{KA}
L.V. Kantorovich and G.P. Akilov, \emph{Functional Analysis}, Second edition, Pergamon Press,
Oxford-Elmsford, N.Y., 1982.

\bibitem{KrRut}
M.A. Krasnoselskii and Ya.B. Rutickii, \emph{Convex Functions and Orlicz Spaces}, Groningen 1961.

\bibitem{KPS}
S.G.~Krein, Ju.I.~Petunin and E.M.~Semenov, \emph{Interpolation of
Linear Operators}, AMS Translations of Math. Monog. \textbf{54},
Providence, 1982.

\bibitem{LT2}
J.~Lindenstrauss and L.~Tzafriri, \emph{Classical Banach Spaces
II}, Springer-Verlag, 1979.

\bibitem{Lo}
G.G. Lorentz, \emph{An inequality for rearrangements}, Amer. Math. Monthly, \textbf{60}~
1953, 176-179.

\bibitem{M}
J. Musielak, \emph{Orlicz Spaces and Modular Spaces}, Lecture Notes in Mathematics 1034, Springer-Verlag 1983.

\bibitem{MO}
W. Matuszewska and W. Orlicz, \emph{On certain properties of
$\varphi$-functions}, Bull. Acad. Polon. Sci. Ser. Sci. Math. Astronom. Phys. \textbf{8}~(1960), 439--443.

\bibitem{N}
I.P. Natanson, \emph{Theory of functions of a real variable}, Frederik
Unger Publ. Co., New York, 1995.



\bibitem{Roy}
H.L. Royden, \emph{Real Analysis}, Macmillan Publishing Company, New York, 1988.

\bibitem{Z}
A.C. Zaanen, \emph{Integration}, North-Holland Publishing Co., Amsterdam, 1967.




\end{thebibliography}
\end{document}